%% file: draft.tex
\documentclass[onefignum,onetabnum]{siamonline190516}

\input{ex_shared}

\ifpdf
\hypersetup{
  pdftitle={An Example Article},
  pdfauthor={D. Doe, P. T. Frank, and J. E. Smith}
}
\fi

\begin{document}

\maketitle

\begin{abstract}
We propose a multilevel Markov chain Monte Carlo (MCMC) method for the Bayesian inference of random field parameters in PDEs using high-resolution data. Compared to existing multilevel MCMC methods, we additionally consider level-dependent data resolution and introduce a suitable likelihood scaling to enable consistent cross-level comparisons. We theoretically show that this approach attains the same convergence rates as when using level-independent treatment of data, but at significantly reduced computational cost. The convergence analysis focuses on Lipschitz continuous transformations of Gaussian random fields with Mat\'ern covariance structure. These results are illustrated using numerical experiments for a 2D plane stress problem, where the Young's modulus is estimated from discretisations of the displacement field.
\end{abstract}



\bibliographystyle{abbrv}

\section{Introduction}
Many problems in science and engineering consist of estimating some set of input parameters $\vartheta$ in a forward model $\mathcal{F}$, based on observations $u^{\text{obs}} \in \mathbb{R}^N$ of the model output. In the context of this paper, we are interested in the case where $\vartheta$ is infinite-dimensional, denoted as $\vartheta \in \mathbb{R}^{\mathbb{N}}$. Generally speaking, we assume $u^{\text{obs}}$ to be related to $\vartheta$ as
\begin{equation}
\label{eq:bip}
u^{\text{obs}} = \mathcal{F}(\vartheta) + \eta,
\end{equation}
where $\eta$ represents the observational noise. The specific structure of $\eta$ impacts the observations and care should be taken when modelling this noise. The noise model used in this paper will be specified in Section \ref{ch:beammodel}. A common example of model \eqref{eq:bip} is when $\mathcal{F}$ is the projection onto the data space of the solution to some known partial differential equation (PDE), and $\vartheta$ are some unknown input parameters to the PDE. Within the context of this paper, we will consider a 2D plane stress problem, where the equations of linear elasticity are solved using linear finite elements (FE). These equations are introduced in Section \ref{ch:beammodel}.

Solving inverse problems is often done by an optimization strategy, where a value of $\vartheta$ is sought that optimally matches the observations $u^{\text{obs}}$. The alternative approach considered here is Bayesian inference through sampling, where a full posterior parameter distribution conditional on the observations is sought \cite{law_data_2015}, here represented by its density $P(\vartheta | u^{\text{obs}})$. The posterior density can be calculated from two factors. The first is the prior density $\pi(\vartheta)$, which encompasses all information available prior to the measurements and is typically straightforward to calculate. Second is the likelihood $\mathcal{L}(u^{\text{obs}}|\vartheta)$, which contains the information obtained using the observations. A single likelihood evaluation for a given parameter value $\vartheta$ consists of solving the forward problem with $\vartheta$ as input and comparing the output to the available data, taking into account the noise model. These factors are combined in Bayes' rule
\begin{equation}
P(\vartheta | u^{\text{obs}}) = \dfrac{\mathcal{L}(u^{\text{obs}}|\vartheta)\pi(\vartheta)}{P(u^{\text{obs}})} \simeq \mathcal{L}(u^{\text{obs}}|\vartheta)\pi(\vartheta).
\end{equation}
The term $P(u^{\text{obs}})$ in the denominator is called the evidence and is typically infeasible to calculate as it involves integrating over the entire data space. However, as it does not depend on $\vartheta$ this is only a normalising constant and most sampling approaches can avoid calculating it directly.

A widely used sampling technique can be found in Markov chain Monte Carlo (MCMC) methods \cite{chen_dimension-robust_2019, cotter_mcmc_2013, hastings_monte-carlo_1970, law_data_2015}. These methods work by constructing a Markov chain that, after a transient phase, asymptotically samples from the posterior distribution. However, as every individual sample requires a likelihood evaluation, and thus a forward solve of $\mathcal{F}$, the computational cost for a given amount of samples critically depends on the cost of evaluating the forward model. In our setting, each forward model evaluation involves solving the PDE approximately on a finite element grid. When using a high spatial resolution of the computational domain, few samples can be constructed. This results in poor Monte Carlo (MC) estimates of the posterior moments.

To reduce computational complexity, we will start from the multilevel MC methodology \cite{blondeel_p-refined_2020, giles_multilevel_2015, lovbak_multilevel_2021, van_barel_robust_2017}. Here, the problem is first solved on a coarse discretization level, where the forward problem is cheaper to evaluate, allowing many more samples to be generated. This coarse solution is subsequently corrected on finer discretization levels. Computational acceleration is achieved because the correction terms require far fewer samples than a direct approximation of the posterior distribution on the finest level. This methodology has been applied to a variety of cases, both in structural mechanics \cite{angelikopoulos_x-tmcmc_2015, blondeel_p-refined_2020} and in other domains \cite{dodwell_hierarchical_2015, latz_multilevel_2018}. 

The principle of multilevel MC has been used extensively in the development of novel methods for Bayesian inversion. Of particular note is the multilevel MCMC algorithm \cite{dodwell_hierarchical_2015}, which implements the multilevel approach in a Metropolis-Hastings algorithm. In recent years, several alternatives to multilevel MCMC have been proposed. Examples of these methods include multilevel sequential MC \cite{latz_multilevel_2018}, multilevel delayed acceptance \cite{christen_markov_2005, lykkegaard_multilevel_2023} and multilevel dimension-independent likelihood-informed MCMC \cite{cui_multilevel_2024}.

The main contribution of this paper is an adaptation of the multilevel MCMC algorithm \cite{dodwell_hierarchical_2015} for the case of high-resolution observations. In many cases, the data originates from a few discrete sensors and the number of observation points is far lower than the number of nodes in the FE grid, even on the coarsest level in the multilevel hierarchy. However, some imaging techniques, for example Digital Image Correlation \cite{helfrick_3d_2011} or optic fibres \cite{anastasopoulos_one-year_2021}, both used in the field of structural mechanics, may yield thousands of measurement points on a single structure. Incorporating this many observations leads to expensive likelihood evaluations on all levels, limiting the advantage of the multilevel approach.

The method we propose involves a level-dependent calculation of the likelihood, where only a weighted subset of the observations is considered at coarser levels. This reduction in observation points naturally leads to considerably cheaper coarse calculations. We give a theoretical analysis of this adapted method to show that similar, significant gains over the classical single-level MCMC approach can be obtained without loss of accuracy on the posterior estimates. 

The posterior estimates are constructed by modelling $\vartheta$ as a single realisation of a transformed Gaussian random field. There is a vast body of research on random field models for PDE inversion. Commonly studied are lognormal fields with applications in e.g. groundwater flow \cite{charrier_strong_2012,charrier_finite_2013,nobile_multi_2015,teckentrup_further_2013}, though other fields have been studied as well \cite{ali_multilevel_2017,blondeel_p-refined_2020,guth_generalized_2024}. We will focus on Lipschitz continuous transformations of Gaussian fields, which are of specific interest to model uniform and Gamma random fields with applications in structural mechanics \cite{blondeel_p-refined_2020}. We adapt the convergence analysis of \cite{charrier_strong_2012}, originally made for log-normal fields, to fit this situation.

This paper is structured as follows. in Section \ref{ch:beammodel} we outline the setup of the model and show how the plane stress problem fits into this formalism. In Section \ref{ch:mlmcmc} we give an overview of the multilevel MCMC algorithm and show how it can be adapted to include a level-dependent treatment of data. Next, in Section \ref{ch:conv_analysis}, we provide a theoretical convergence analysis of level-dependent data treatment as well as a relaxation of some assumptions made on the multilevel MCMC methodology. Finally, Section \ref{ch:numerics} contains a numerical showcase of the efficiency of the method for the clamped beam setup.

\section{Model setup}
\label{ch:beammodel}

In this Section, we outline the model problem from structural mechanics used to illustrate our MCMC method. In Subsection \ref{sec:forward}, we give an overview of the physical properties of the model problem under investigation. Next, in Subsection \ref{sec:highresobs}, we illustrate how high-resolution observations typically fit in this model. Finally, Subsection \ref{sec:KL} describes how the spatial variation of the stiffness is modelled by a Karhunen-Lo\`eve expansion.

\subsection{Forward problem}
\label{sec:forward}

Throughout this paper, we will consider a 2D forward problem $\mathcal{F}(\vartheta)$ that consists of solving the equations of linear elasticity in the plane stress case. In these equations, the parameter $\vartheta$ we seek is the spatially varying Young's modulus $E$ of the material throughout the domain. The Poisson's ratio $\tilde{\nu}$ is assumed to be constant and below the limit value of $\tilde{\nu} = 0.5$ where the material is incompressible. In order to write the linear elastic equations in their simplest forms, it is easiest to use  the Lam\'e parameters $\tilde{\lambda}$ and $\tilde{\mu}$, that relate to $\vartheta$ and $\tilde{\nu}$ as 

\begin{equation}
\tilde{\lambda} = \dfrac{\vartheta\tilde{\nu}}{(1+\tilde{\nu})(1-2\tilde{\nu})}, \qquad
\tilde{\mu} = \frac{\vartheta}{2(1+\tilde{\nu})}.
\end{equation}

We can thus view these as spatially varying functions $\tilde{\lambda},\tilde{\mu}:\mathbb{R}^2\to\mathbb{R}$. With the Lam\'e parameters, the equations of linear elasticity can be combined into one equation of the following form:

\begin{equation}
-\nabla \cdot \left[\tilde{\lambda}(\nabla\cdot u)I + \tilde{\mu}(\nabla u + (\nabla u)^T)\right] = F_{\text{body}}.
\label{eq:lin_elastic}
\end{equation}

In this equation, the term $F_{\text{body}}:\mathbb{R}^2\to \mathbb{R}^2$ contains the body forces acting in the system, and $I$ denotes the identity matrix. Equation \eqref{eq:lin_elastic} links the spatially varying Young's modulus to the displacement field $u:\mathbb{R}^2\to \mathbb{R}^2$. For practical calculations, we discretize this PDE on a finite element grid. We consider linear rectangular elements for the discretization of the underlying parameters \cite{francois_stabil_2021}.

We will solve the equations of linear elasticity in a relatively simple 2D model problem consisting of a clamped beam, as shown schematically in figure \ref{fig:beam}. A line load is place along the center third of the top edge of the beam. We assume that the concrete is linearly elastic and isotropic. This model allows damage to the structure to be modelled as a reduction of the Young's modulus. More detailed values of material parameters are given in Section \ref{ch:numerics}.

\begin{figure}
\centering
\resizebox{\textwidth}{!}{
\begin{tikzpicture}
\fill[color={rgb,255:red,231; green,238; blue,242}] (0,0) rectangle (15, 3);
\foreach \x in {1,...,14}
	{\draw (\x,0) -- (\x,3);
	\fill[color={rgb,255:red,60; green,110; blue,58}] (\x, 0) circle (0.07);
	\fill[color={rgb,255:red,60; green,110; blue,58}] (\x, 3) circle (0.07);}
\foreach \y in {0.25,0.5,0.75,1,1.25,1.5,1.75,2,2.25,2.5,2.75}
	{\draw (0, \y) -- (15, \y);}
\draw[color={rgb,255:red,60; green,110; blue,58}, line width=0.5mm] (0,0) rectangle (15, 3);
\draw[ultra thick] (0,-0.5) -- (0, 3.6);
\draw[ultra thick] (15,-0.5) -- (15, 3.6);
\foreach \y in {0,...,18}
	{\draw (-0.5, \y/5) -- (0, \y/5-0.5);
	\draw (15, \y/5) -- (15.5, \y/5-0.5);}
\draw (-0.5, -0.2) -- (-0.2, -0.5);
\draw (-0.5, -0.4) -- (-0.4, -0.5);
\draw (-0.3, 3.6) -- (0, 3.3);
\draw (-0.1, 3.6) -- (0, 3.5);
\draw (15, -0.2) -- (15.3, -0.5);
\draw (15, -0.4) -- (15.1, -0.5);
\draw (15.2, 3.6) -- (15.5, 3.3);
\draw (15.4, 3.6) -- (15.5, 3.5);
\foreach \x in {10.5,...,19.5}
	{\draw[->, ultra thick] (\x/2, 3.6) -- (\x/2, 3);}
\end{tikzpicture}%
}
\caption{Beam clamped at both ends. The beam is discretized using rectangular finite elements, and observations are made at element nodes (green dots) along the green lines.}
\label{fig:beam}
\end{figure}
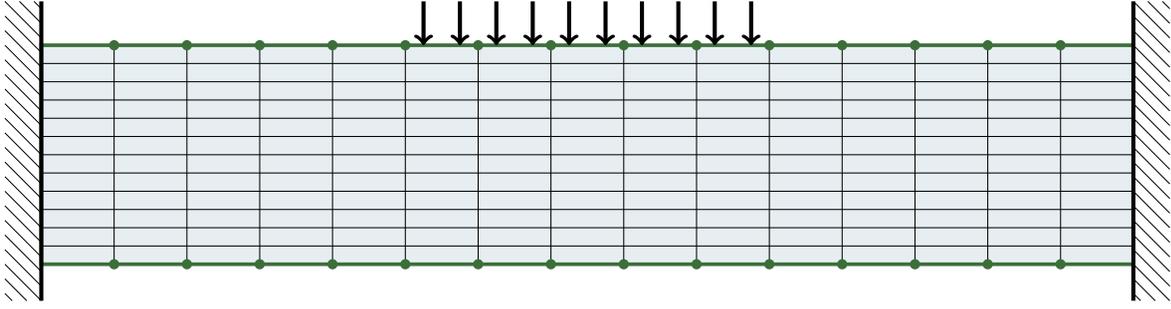

\subsection{High-resolution observations}
\label{sec:highresobs}

The output of the forward model consists of the 2D displacement field $u$ of the beam. We only consider the displacements along the edge of the beam, as shown by the green lines in figure \ref{fig:beam}. The observations are assumed to be made at very high resolution \cite{helfrick_3d_2011}, which is modelled by assuming that we have an observation at each element node along the edges of the beam. By constructing a finite element grid with a sufficient number of elements, this represents the physical case well. As such, we write
\begin{equation} 
\tilde{u} = \left\lbrace u(x_i, y_i)\right\rbrace_{(x_i, y_i)\in \mathcal{N}^{\text{grid}}},
\end{equation}
where $\mathcal{N}^{\text{grid}}$ is the set of finite element nodes along the edges of the beam, shown by green dots in figure \ref{fig:beam}. The observations $u^{\text{obs}}$ are then defined as $u^{\text{obs}} = \tilde{u} + \eta$, with noise $\eta$. For the purpose of this paper, the observations are treated as distinct, independent, normally distributed measurements such that $\eta \sim \mathcal{N}(0, \sigma_F^2I_{2N})$, where $\sigma_F^2$ is called the fidelity and $N = \#\mathcal{N}_{\text{grid}}$. As each observation point has a horizontal and vertical component of the displacement, which are assumed to have independent noise, there are $2N$ measurement errors. 
\begin{remark}
Observation techniques that yield very high-resolution data will typically show some spatial correlation in measurement errors. Exploiting these correlation structures through a more complex noise model might improve the efficiency of the multilevel algorithm \cite{simoen_prediction_2013}, however such a treatment is beyond the scope of the current paper.
\end{remark}

\subsection{Karhunen-Lo\`eve representation}
\label{sec:KL}

The inverse problem under study is interested in recovering the entire stiffness field. This will be modelled as a random field $a(x,\omega)$, where $x$ denotes the spatial and $\omega$ the stochastic dependence. The ground truth value is then defined as a realisation of $a$ for one specific value of $\omega$. There are many ways of modelling random fields \cite{liu_advances_2019}, but we will focus the Karhunen-Lo\`eve (KL) expansion, which is a well-established approach, particularly in the context of multilevel Monte Carlo methods \cite{charrier_strong_2012,charrier_finite_2013,teckentrup_further_2013}.

\subsubsection{KL expression of a Gaussian field}

One of the simplest random fields to model is a Gaussian field $g$ \cite{cliffe_multilevel_2011}. Here, the KL expansion takes the form 
\begin{equation}
g(x,\omega) = \mathbb{E}[g(\cdot,\omega)] + \sum_{m=1}^{\infty}\sqrt{\lambda_m}\xi_m(\omega)b_m(x).
\label{eq:normalfield}
\end{equation}
In this expression, $\{\lambda_m\}_{m\in\mathbb{N}}$ are the eigenvalues and $\{b_m\}_{m\in\mathbb{N}}$ the normalized eigenfunctions of the covariance kernel $C(x,y)$ of $g$, obtained by solving the eigenvalue problem \cite{blondeel_p-refined_2020}
\begin{equation}
\int_D C(x,y)b_m(y)dy = \lambda_m b_m(x).
\end{equation} 
The $\{\xi_m\}_{m\in\mathbb{N}}$ in equation \eqref{eq:normalfield} are a set of i.i.d. standard normal variables. In this setting, the problem of estimating $\vartheta$ becomes an estimation of the realisations of the $\{\xi_m\}_{m\in\mathbb{N}}$ that form the desired stiffness field. For practical calculations, we will consider a finite-dimensional approximation to $g$ by truncating the series after $M$ terms. We construct the KL terms using as computational domain the unit square $D = [0,1]\times [0,1]$, which can be easily mapped to the dimensions of the beam. We model the normal field using a Mat\'ern covariance kernel
\begin{equation}
\label{eq:covariance}
C(x,y) = \dfrac{\sigma^2}{2^{\nu -1}\Gamma(\nu)}\left(\sqrt{2\nu}\dfrac{\|x-y\|_2}{\lambda}\right)^{\nu}K_{\nu}\left(\sqrt{2\nu}\dfrac{\|x-y\|_2}{\lambda}\right),
\end{equation}
where $\sigma^2$ is the variance of the random field, $\lambda$ the correlation length scale, $\nu$ the smoothness parameter and $K_{\nu}$ the modified Bessel function of the second kind. 

\subsubsection{Transformation to other fields}

Non-Gaussian fields can be modelled by introducing a transformation of the Gaussian field, depending on the desired properties of the field. One common example for this is an exponential transformation of $g(x,\omega)$, resulting in a log-normal field \cite{charrier_strong_2012,charrier_finite_2013,dodwell_hierarchical_2015,teckentrup_further_2013}. In the context of linear elastic models in structural mechanics, another field that is regularly used is a Gamma random field \cite{ahmadian_regularisation_1998, lieven_finiteelement_2001, titurus_regularization_2008, weber_consistent_2009}, which can be obtained through the following transformation \cite{blondeel_p-refined_2020}
\begin{equation}
a(x,\omega) = \mu \gamma^{-1}\left[\kappa, \dfrac{\Gamma(\kappa)}{2}\left(1+\text{erf}\left(\dfrac{g(x,\omega)}{\sqrt{2}}\right)\right)\right],
\end{equation}
where $\Gamma$ denotes the Gamma function, $\gamma^{-1}$ the inverse of the lower incomplete gamma function and erf the error function. $\mu$ and $\kappa$ are the scale and shape parameters of the Gamma field, respectively. In order to obtain the desired convergence properties of the MCMC estimator, the random field $a$ must be sufficiently bounded away from zero. Unfortunately, this is not the case for (among others) the Gamma random field. A practical workaround is to instead consider an approximation of this field
\begin{equation}
a_{\varphi}(x,\omega) = \varphi\cdot\dfrac{\exp(g(x,\omega))}{1+\exp(g(x,\omega))} + \mu \gamma^{-1}\left[\kappa, \dfrac{\Gamma(\kappa)}{2}\left(1+\text{erf}\left(\dfrac{g(x,\omega)}{\sqrt{2}}\right)\right)\right],
\label{eq:transformation}
\end{equation}
for a well-chosen value of $\varphi > 0$. Probability density functions for a random variable defined in this way are shown in Figure \ref{fig:densities} for different values of $\varphi$. In Proposition \ref{thm:bound_afield}, we will prove that this approximation is indeed sufficiently bounded away from zero for a broad class of transformed fields $a$, for any $\varphi > 0$.

\begin{figure}
\centering
\includegraphics[width=\textwidth]{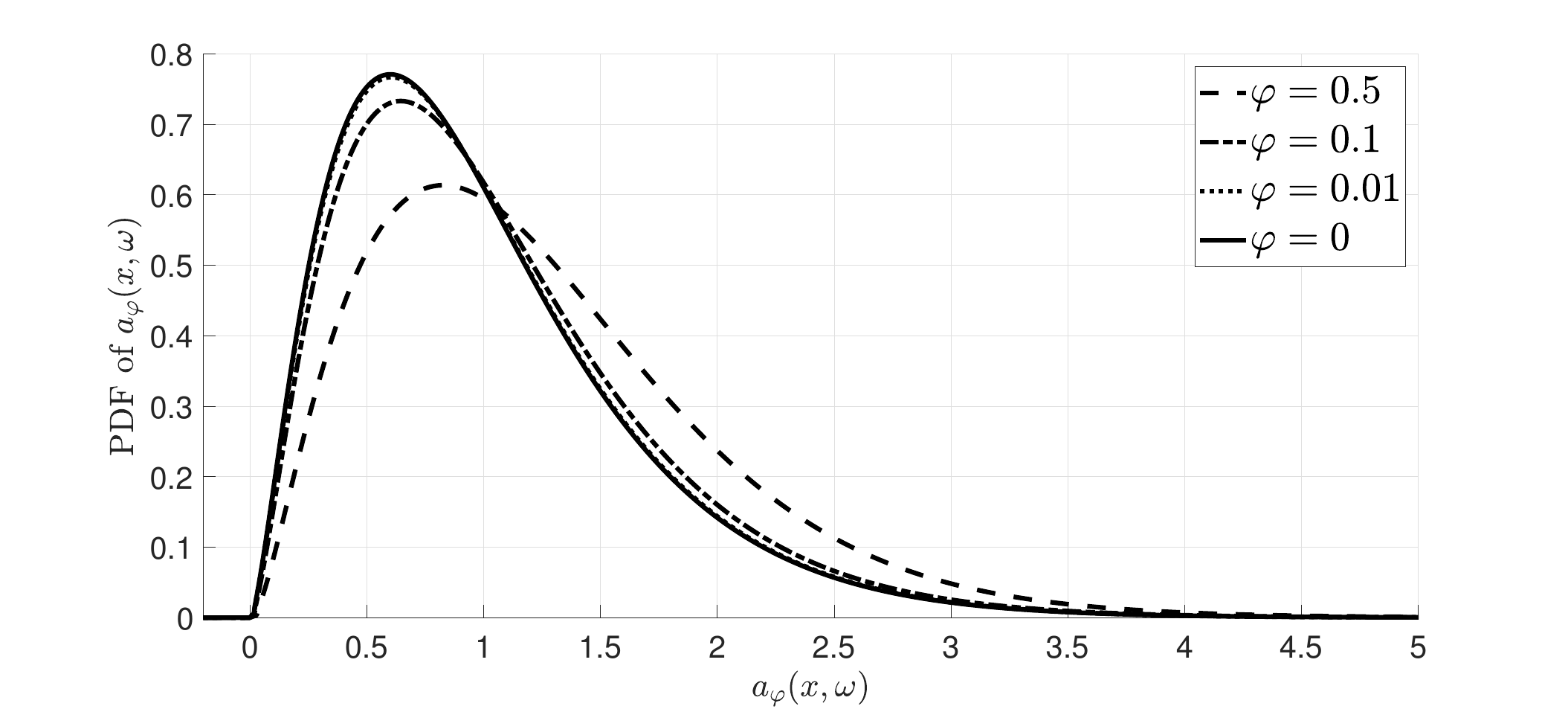}
\caption{Probability density functions of random variables constructed through the transformation in Equation \eqref{eq:transformation}, for several values of $\varphi$. Shape and scale parameters are chosen as $\kappa = 2.5$ and $\mu = 0.4$, respectively.}
\label{fig:densities}
\end{figure}

\section{Multilevel Markov chain Monte Carlo with full-field data}
\label{ch:mlmcmc}

In this section, we will first give an overview of the multilevel MCMC algorithm in Subsection \ref{sec:mlmcmc}, to introduce the basic concepts and notation. Next, in Subsection \ref{sec:leveldata}, we will show how the algorithm can be adapted to include a level-dependent treatment of data to enable the use of full-field measurement techniques. Finally, Subsection \ref{sec:comparison} contains some notes on a qualitative comparison between our method and other improved variants of multilevel MCMC.

\subsection{Multilevel MCMC}
\label{sec:mlmcmc}

Because the likelihood $\mathcal{L}(u^{\text{obs}}|\vartheta)$ encompasses information from the model output, it depends on the discretization level $R$ of the system. Similarly, the true underlying infinite-dimensional parameter, $\vartheta \in \mathbb{R}^{\mathbb{N}}$ can in practice only be approximated as a finite-dimensional parameter $\theta \in \mathbb{R}^M$. For clarity, we will include these discretization levels in the relevant terms and write the practical version of Bayes' theorem as
\begin{equation}
P_{M,R}(\theta | u^{\text{obs}}) \simeq \pi_M(\theta)\mathcal{L}_R(u^{\text{obs}}|\theta).
\end{equation}

\subsubsection{Underlying idea}
\label{subsec:underlying}

The multilevel Monte Carlo approach addresses the problem of prohibitively expensive forward models by exploiting a hierarchy of discretisation levels $\ell = 0,\ldots L$ \cite{giles_multilevel_2008, giles_multilevel_2015}. Throughout this paper, quantities that are level-dependent will be given a subscript $\ell$. In the simple setting discussed earlier, we assumed to have a parameter $\theta$ of dimension $M$, with the PDE of Equation \eqref{eq:lin_elastic} discretized on a finite element grid with discretization parameter $R$, representing the number of finite elements on the grid. As mentioned above, we will write these parameters as $M_{\ell}$ and $R_{\ell}$ respectively, as they can change between levels. The multilevel Monte Carlo methodology allows us to estimate moments of the posterior distributions $\nu_{M_{\ell}, R_{\ell}}$ for all values of $\ell$, which converge to the true underlying posterior as $M_{\ell}, R_{\ell} \to \infty$ \cite{teckentrup_further_2013}. In this context, $\{R_{\ell}\}_{\ell=0}^L$ is an increasing sequence, where we assume, for simplicity, that there exists an $s \geq 1$ such that
\begin{equation}
\label{eq:refinefactor}
R_{\ell} = sR_{\ell - 1}.
\end{equation}
Denoting by $d$ the physical dimension of the problem, a simple choice for the sequence $\{R_{\ell}\}_{\ell=0}^L$ is constructing it such that $s=2^d$, but other choices for $s$ can be made as well. Similarly, $\{M_{\ell}\}_{\ell=0}^L$ is also an increasing sequence. When we regard $M_{\ell}$ as the truncation number of the KL expansion on level $\ell$, this implies that $\{\theta_{\ell}\}_{\ell=0}^L$ is a nested sequence, i.e. $\theta_{\ell-1} \subseteq \theta_{\ell}$.

Having access to different discretisation levels, the multilevel Monte Carlo approach aims to find an estimator for a quantity of interest $Q$ by first constructing an approximate posterior estimate using a coarse approximation to the forward model, and then correcting this estimate using relatively few samples on finer levels. Using two different but correlated chains $\theta_{\ell}$ and $\Theta_{\ell-1}$, it does so by using a telescopic sum:
\begin{equation}
\label{eq:telescopic}
\widehat{Q}^{\text{MLMC}}(\theta) = \dfrac{1}{N_0}\sum_{n=B_0+1}^{B_0+N_0}Q_0(\theta^n_0) + \sum_{\ell=1}^L \dfrac{1}{N_{\ell}} \sum_{n = B_{\ell}+1}^{B_{\ell}+N_{\ell}} \left[Q_{\ell}(\theta^n_{\ell}) - Q_{\ell-1}(\Theta^n_{\ell-1})\right].
\end{equation}
The first term in this telescopic sum is a Monte Carlo estimate calculated with the classical Metropolis-Hastings algorithm \cite{hastings_monte-carlo_1970}, with the subscript 0 added to emphasize the fact that this estimate is calculated on the coarsest mesh. The later terms are Monte Carlo estimates for the corrections on the initial coarse estimate, calculated using a chain $\theta_{\ell}$ on level $\ell$ and a correlated chain $\Theta_{\ell-1}$ on the previous level $\ell -1$. In this notation, we will occasionally write $Y_{\ell} = Q_{\ell} - Q_{\ell-1}$. Each term in this sum is made using $N_{\ell}$ samples, as the first $B_{\ell}$ samples of each chain are discarded as burn-in. Every additional term in this sum contains evaluations of the forward model on a finer level, and as such fewer samples can be generated due to the increasing cost. However, the crucial observation is that as the discretization level $\ell$ increases, the estimates get closer to the true underlying value of $Q(\theta)$, so the magnitude of the correction terms decreases \cite{giles_multilevel_2015}. As such, fewer samples are needed to reach a target threshold of the Monte Carlo error for the later correction terms, leading to a cheap estimate of the correction terms on all levels and a significant decrease in cost over the single-level approach.

\begin{algorithm}
\caption{Multilevel Markov chain Monte Carlo}
\label{alg:mlmcmc}
\begin{algorithmic}
\STATE{Choose initial state $\theta_L^0 \in \mathbb{R}^{R_L}$. For every $\ell \in 0,\ldots, L$ the first $R_{\ell}$ modes of $\theta_L^0$ are the initial state $\theta_{\ell}^0$ on level $\ell$. Choose subsampling rates $\tau_{\ell}$ for $\ell=1,\ldots,L$. Set level iteration numbers $n_{\ell}=0$ for $\ell=0,\ldots,L-1$}
\FOR{$n_0 = 0,\ldots, N_0$}
\STATE{Generate pCN proposal move $\theta'_0 \sim q_0(\cdot|\theta_0^{n_0})$}
\STATE{Evaluate pCN acceptance criterion
\[
\alpha_{M_0,R_0} = \text{min}\left\lbrace 1, \dfrac{\mathcal{L}_0(u^{\text{obs}}|\theta_0')}{\mathcal{L}_0(u^{\text{obs}}|\theta_0^{n_0})}\right\rbrace
\]}
\STATE{With probability $\alpha_{M_0,R_0}$ accept proposal and set $\theta_0^{n_0+1} = \theta'_0$,}
\STATE{Otherwise reject proposal and set $\theta_0^{n_0+1} = \theta_0^{n_0}$}
\FOR{$\ell = 1,\ldots,L$}
\IF{$n_{\ell-1}\equiv 0 \mod \tau_{\ell}$}
\STATE{Generate proposal coarse modes by subsampling chain on level $\ell - 1$: $(\theta_{\ell}^C)' = \theta_{\ell-1}^{n_{\ell}\tau_{\ell}}$}
\STATE{Generate proposal fine modes using pCN proposal distribution: $(\theta_{\ell}^F)' \sim q_{\ell}(\cdot|\theta_{\ell}^{n_{\ell},F})$}
\STATE{Set $\theta_{\ell}' = [(\theta_{\ell}^C)', (\theta_{\ell}^F)']$ and calculate
\[
\alpha_{M_{\ell}, R_{\ell}} = \min\left\lbrace1, \dfrac{\mathcal{L}_{\ell}(u^{\text{obs}}|\theta_{\ell}')\mathcal{L}_{\ell-1}(u^{\text{obs}}|\theta_{\ell}^{n,C})}{\mathcal{L}_{\ell}(u^{\text{obs}}|\theta_{\ell}^n)\mathcal{L}_{\ell-1}(u^{\text{obs}}|(\theta_{\ell}^C)')}\right\rbrace
\]}
\STATE{With probability $\alpha_{M_{\ell},R_{\ell}}$ accept proposal and set $\theta_{\ell}^{n_{\ell}+1} = \theta_{\ell}'$,}
\STATE{Otherwise reject and set $\theta_{\ell}^{n_{\ell}+1} = \theta_{\ell}^{n_{\ell}}$}
\STATE{$n_{\ell} \leftarrow n_{\ell}+1$}
\ENDIF
\ENDFOR
\ENDFOR
\RETURN{Markov chains of elements $\{\theta^n_{\ell}\}$ at different discretisation levels}
\end{algorithmic}
\end{algorithm}

\subsubsection{Subsampling Markov chains}

The multilevel Markov chain Monte Carlo algorithm \cite{dodwell_hierarchical_2015} is used to generate the chains $\theta_{\ell}$ and $\Theta_{\ell-1}$ in order to consistently compare them and estimate the correction term on level $\ell$. In theory, the method requires independent samples from  $\nu_{M_{\ell-1}, R_{\ell-1}}$ to construct $\Theta_{\ell-1}$. Unfortunately it is impossible to generate these, but we do have access to a Markov chain of elements $\theta_{\ell-1}$ which, provided this chain has been carried out long enough, samples from $\nu_{M_{\ell-1}, R_{\ell-1}}$. The samples in this chain are correlated, but by subsampling the chain $\theta_{\ell-1}$ at a sufficiently high rate we get approximately uncorrelated samples that can be used to construct the chain $\Theta_{\ell-1}$. We thus define a subsampling rate $\tau_{\ell}$ on each level, and define each sample $\Theta_{\ell-1}^n$ as $\theta_{\ell-1}^{n\tau_{\ell}}$.

\subsubsection{Proposal moves}

Having access to the subsampled chain $\Theta_{\ell-1}$, we can use this to construct the finer level chain $\theta_{\ell}$. Recalling that $\{\theta_{\ell}\}_{\ell=0}^L$ is a nested sequence, we can split $\theta_{\ell}$ in two parts by writing $\theta_{\ell} = [\theta_{\ell}^C, \theta_{\ell}^F]$. $\theta_{\ell}^C$ contains the first $M_{\ell - 1}$ modes of $\theta_{\ell}$, which we call the coarse modes. $\theta_{\ell}^F$ contains the other $M_{\ell} - M_{\ell - 1}$ modes, and these are called the fine modes. For consistency, we define $M_{-1}=0$. Starting from a fine-level sample $\theta_{\ell}^n$, we can generate the next element of the chain on this level in a similar way as the Metropolis-Hastings algorithm: we first propose a new move, which is constructed as $\theta_{\ell}' = [\Theta_{\ell-1}^{n+1}, (\theta_{\ell}^F)']$. The coarse modes are thus proposed by subsampling the chain on the previous level as discussed above, while the fine modes are proposed using a classical proposal distribution $q$ centred at $\theta_{\ell}^{n, F}$. These coarse and fine modes together constitute the proposal move on level $\ell$. This proposal is then compared to an acceptance criterion, which is the multilevel equivalent of the Metropolis-Hastings acceptance criterion:
\begin{equation}
\alpha_{M_{\ell}, R_{\ell}} = \min\left\lbrace1, \dfrac{P_{M_{\ell}, R_{\ell}}(\theta_{\ell}')q_{\ell}(\theta_{\ell}^n|\theta_{\ell}')}{P_{M_{\ell}, R_{\ell}}(\theta_{\ell}^n) q_{\ell}(\theta_{\ell}'|\theta_{\ell}^n)}\right\rbrace,
\end{equation}
where $P_{M_{\ell}, R_{\ell}}$ is the posterior density on level $\ell$ and $q_{\ell}$ is the proposal move discussed above. 

As proposal for the fine modes, we will use the preconditioned Crank-Nicholson (pCN) proposal \cite{cotter_mcmc_2013}
\begin{equation}
q((\theta_{\ell}^F)'|\theta_{\ell}^{n,F}) = \theta_{\ell}^{n,F}\sqrt{1-\beta^2} + \beta\zeta,
\end{equation}
with $\beta \in [0, 1]$ and $\zeta \sim \mathcal{N}(0, \mathcal{C})$, where $\mathcal{C}$ is the prior covariance in the case of a normally distributed prior. The pCN proposal distribution is essentially dimension-independent and as such is very well suited for high-dimensional problems such as the one studied in this paper. When using the pCN proposal distribution for the fine modes of the proposal, the acceptance criterion can be written as \cite{dodwell_hierarchical_2015}
\begin{equation}
\alpha_{M_{\ell}, R_{\ell}} = \min\left\lbrace1, \dfrac{\mathcal{L}_{\ell}(u^{\text{obs}}|\theta_{\ell}')\mathcal{L}_{\ell-1}(u^{\text{obs}}|\theta_{\ell}^{n,C})}{\mathcal{L}_{\ell}(u^{\text{obs}}|\theta_{\ell}^n)\mathcal{L}_{\ell-1}(u^{\text{obs}}|(\theta_{\ell}^C)')}\right\rbrace.
\label{eq:ml_accept}
\end{equation}
The multilevel Markov chain Monte Carlo algorithm is summarized in algorithm \ref{alg:mlmcmc}.

\subsection{Incorporation of level-dependent data}
\label{sec:leveldata}

The level-dependent treatment of data we propose focuses specifically on the calculation of the likelihood function. For simplicity, we assume that the measurement errors are independent and $\mathcal{N}(0, \sigma_F^2I_{2N_L})$-distributed, where $N_L = \#\mathcal{N}_L^{\text{grid}}$ is large but finite. For practical purposes, this corresponds to constructing the finest level mesh in such a way that its nodes coincide with the observation points. In this setting, a single-level likelihood takes the form
\begin{equation}
\mathcal{L}(u^{\text{obs}}|\theta) \simeq \exp\left[\dfrac{-\|u^{\text{obs}} - F(\theta)\|^2}{2\sigma_F^2}\right].
\end{equation}
\begin{remark}
The above expression is ill-suited if we would consider the limit for infinite-dimensional data \cite{kahle_bayesian_2019,stuart_inverse_2010}. While a more general treatment of continuous (and thus, infinite-dimensional and spatially correlated) data is of particular interest to us, this is left for future work and the scope of the current paper will be limited to the case of high-, albeit finite-dimensional data. To fit inside the existing multilevel MCMC framework \cite{dodwell_hierarchical_2015}, the grid resolution $R_{\ell}$ and KL truncation $M_{\ell}$ will still be allowed to enter the asymptotic regime of $\ell \to \infty$ for the convergence analysis, but $N_L$ is considered to reach a maximum value at finite level $L$.
\end{remark}
For the multilevel approach, we assume to have a nested sequence of meshes, i.e. the mesh nodes on level $\ell - 1$ are a subset of the nodes on level $\ell$. The level-dependent treatment then consists of considering at each level a data vector consisting only of those measurement points that coincide with the mesh nodes at that level, either by selecting only the observations at those nodes or by applying a general dimension-reducing weighting function to all observations. To accentuate this level-dependent treatment, we will include a subscript $\ell$ in the notation of the data vector $u^{\text{obs}}$. This approach is schematically shown in Figure \ref{fig:datascheme}, next to a level-independent case with low-resolution data.

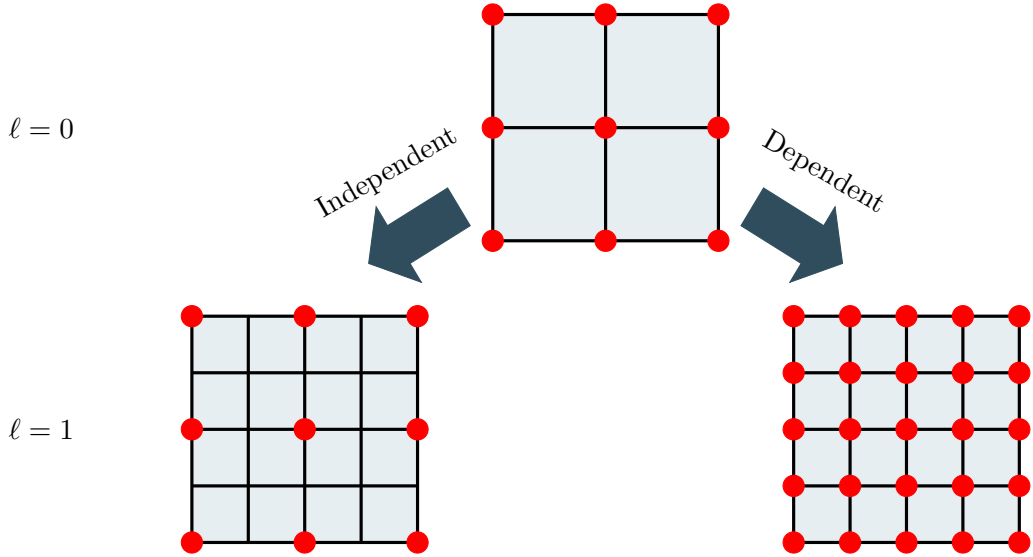
\begin{figure}
\flushleft
\hspace*{1mm}
\begin{tikzpicture}
\fill[color={rgb,255:red,231; green,238; blue,242}] (0,0) rectangle (3,3);
\fill[color={rgb,255:red,231; green,238; blue,242}] (8,0) rectangle (11,3);
\fill[color={rgb,255:red,231; green,238; blue,242}] (4,4) rectangle (7,7);
\draw[very thick] (0,0) rectangle (3,3);
\draw[very thick] (8,0) rectangle (11,3);
\draw[very thick] (4,4) rectangle (7,7);
\draw[very thick] (5.5,4) -- (5.5,7);
\draw[very thick] (4,5.5) -- (7,5.5);
\draw[very thick] (0.75,0) -- (0.75,3);
\draw[very thick] (1.5,0) -- (1.5,3);
\draw[very thick] (2.25,0) -- (2.25,3);
\draw[very thick] (0,0.75) -- (3,0.75);
\draw[very thick] (0,1.5) -- (3,1.5);
\draw[very thick] (0,2.25) -- (3,2.25);
\draw[very thick] (8.75,0) -- (8.75,3);
\draw[very thick] (9.5,0) -- (9.5,3);
\draw[very thick] (10.25,0) -- (10.25,3);
\draw[very thick] (8,0.75) -- (11,0.75);
\draw[very thick] (8,1.5) -- (11,1.5);
\draw[very thick] (8,2.25) -- (11,2.25);
\foreach \x in {4,5.5,7}
	\foreach \y in {4,5.5,7}
		\filldraw [red] (\x, \y) circle (4pt);
\foreach \x in {0,1.5,3}
	\foreach \y in {0,1.5,3}
		\filldraw [red] (\x, \y) circle (4pt);
\foreach \x in {8,8.75,9.5,10.25,11}
	\foreach \y in {0,0.75,1.5,2.25,3}
		\filldraw [red] (\x, \y) circle (4pt);
\filldraw[color={rgb,255:red,47; green,77; blue,93}] (3.4,4.7) -- (3.7,4.2) -- (2.9, 3.7) -- (2.6, 4.2) -- cycle; 
\filldraw[color={rgb,255:red,47; green,77; blue,93}] (2.45,4.45) -- (3.05,3.45) -- (2.35,3.7) -- cycle; 
\filldraw[color={rgb,255:red,47; green,77; blue,93}] (7.6,4.7) -- (7.3,4.2) -- (8.1, 3.7) -- (8.4, 4.2) -- cycle; 
\filldraw[color={rgb,255:red,47; green,77; blue,93}] (8.55,4.45) -- (7.95,3.45) -- (8.65,3.7) -- cycle; 
\node[rotate=30] at (2.6,4.9) {Independent};
\node[rotate=330] at (8.4,4.9) {Dependent};
\node at (-2, 5.5) {$\ell = 0$};
\node at (-2, 1.5) {$\ell = 1$};
\end{tikzpicture}
\caption{Two-level sketch of the difference between classical level-independent treatment of data used for low-resolution observations and level-dependent treatment of data for high-resolution observations. Red dots indicate observations used in each case.}
\label{fig:datascheme}
\end{figure}

In order to accommodate for the increasing number of observations used between levels in the calculation of the likelihood, a weighting term should be included in the expression of the likelihood such that convergence of the vector norms of elements on increasingly fine levels is achieved. This will ensure a consistent cross-level comparison between likelihood scores, which is of vital importance as the multilevel acceptance criterion in Equation \eqref{eq:ml_accept} contains likelihood evaluations at both levels $\ell$ and $\ell - 1$.

For the level-dependent likelihood function, not only are $\theta$ and the discretized forward model $F$ level-dependent, but we also now consider the level-dependent nature of the data. In order to account for the increase in observations considered at finer levels, we thus weigh the likelihood accordingly:
\begin{equation}
\mathcal{L}_{\ell}(u_{\ell}^{\text{obs}}|\theta_{\ell}) \simeq \exp\left[\dfrac{-\|W_{\ell}(u^{\text{obs}}) - F_{\ell}(\theta_{\ell})\|^2}{2N_{\ell}\sigma_F^2}\right],
\label{eq:levellike}
\end{equation}
with $N_{\ell}$ the amount of observation points used at each level. Here $W_{\ell}: \mathbb{R}^{N_L}\to\mathbb{R}^{N_{\ell}}$ is a weighting function that maps all observations to the nodes used on the mesh at level $\ell$. In the case where we simply select only those observations coinciding with the coarser mesh nodes, we write $W_{\ell} = W_{\ell}^S$. For consistency, we additionally define $W_{\ell}$ to be the identity function on $\mathbb{R}^{N_L}$ for $\ell \geq L$.
\begin{remark}
The inclusion of $N_{\ell}$ in Equation \eqref{eq:levellike} is necessary for consistency of the level-dependent likelihood expressions. Effectively, it treats the observations as if their noise level is lower at coarse levels and higher at fine levels. This can be seen as considering the observation on the coarse-level node to be a local average of all the neighbouring observations on the finest level. With an appropriate choice of $W_{\ell}$, being an actual local average, this interpretation can be made literal.
\end{remark}

\subsection{Some thoughts on comparison to improved multilevel MCMC methods}
\label{sec:comparison}
The study in this paper is concerned with an extension of the multilevel MCMC algorithm \cite{dodwell_hierarchical_2015} to include high-resolution data. In recent years, many alternative, improved methods have been developed that combine the ideas of multilevel Monte Carlo with MCMC approaches. Most of these methods focus on improving the sampling efficiency, i.e. a reduction of the required number of samples at the finest level. Two notable examples are the multilevel Delayed Acceptance \cite{lykkegaard_multilevel_2023} and multilevel Dimension-Independent Likelihood-Informed MCMC \cite{cui_multilevel_2024} algorithms. Our method focuses on an efficient likelihood evaluation, i.e. a reduction of the cost per sample, in situations with high-resolution data. In Sections \ref{ch:conv_analysis} and \ref{ch:numerics} our method is only compared to the standard multilevel MCMC algorithm. 
As the novelty of our approach consists of selecting weighted subsets of observations to construct approximate likelihoods at coarser levels, we expect our suggested strategy to be straightforwardly applicable in those methods as well, with similar gains.

\section{Convergence analysis}
\label{ch:conv_analysis}

The multilevel MCMC algorithm has been shown to have a considerably lower computational cost than classical single-level methods \cite{dodwell_hierarchical_2015}. In this Section, we will adapt the existing convergence analysis to the specifics of our method. In Subsection \ref{seq:analysis_field}, we adapt existing analyses for strong convergence of an elliptic PDE solution based on truncated KL expansions of the underlying random field \cite{charrier_strong_2012,charrier_finite_2013,teckentrup_further_2013} to our framework. Next, in Subsection \ref{seq:analysis_data}, we will show that the bound on the $\varepsilon$-cost of the multilevel MCMC algorithm obtained by \cite{dodwell_hierarchical_2015} can be generalised to a level-dependent treatment of data and that a similar bound on the $\varepsilon$-cost holds for more general random fields. For brevity, we will not repeat the full convergence analysis, but rather only show how the existing propositions can be altered where necessary.

\subsection{Strong convergence of PDE solution}
\label{seq:analysis_field}

The bound on the $\varepsilon$-cost of the multilevel MCMC algorithm depends on two aspects of the model approximation, being the truncation of the KL expansion after a finite number of terms and the finite element discretisation. The effect of the latter is a well-studied property of the multilevel Monte Carlo methodology, and the existing analyses of this carry over to our case without any alteration needed \cite{charrier_finite_2013,teckentrup_further_2013}. To study the effect of truncating the KL expansion, the convergence analysis of \cite{charrier_strong_2012} will be tweaked to show that similar convergence properties hold for transformations such as in Equation \eqref{eq:transformation} (for the case of a Gamma random field), for any $\varphi > 0$. The analysis is made in terms of the Sobolev regularity $k$ corresponding to the Sobolev space $H^k$. This includes Mat\'ern covariances, for which $k = 2\nu + d$ \cite{bachmayr_representations_2016,graham_quasi-monte_2015,lord_introduction_2014}.

\subsubsection{Technical lemmas}

In order to prove the strong convergence of the PDE solution, we require the following assumptions on the eigenpairs $(\lambda_m, b_m)_{m\geq 1}$ of the covariance operator of the Gaussian field:
\begin{enumerate}[label=B\arabic*.]
\item The eigenfunctions $b_m$ are continuously differentiable
\item The series $\sum_{m\geq 1}\lambda_m\|b_m\|^2_{L^{\infty}(D)}$ is convergent
\item There exists a $\beta \in (0,1)$ such that the series $\sum_{m\geq 1}\lambda_m\|b_m\|^{2(1-\beta)}_{L^{\infty}(D)}\|\nabla b_m\|^{2\beta}_{L^{\infty}(D)}$ is convergent
\end{enumerate}

The analysis of \cite{charrier_strong_2012} shows that these assumptions hold in the case of an exponential two-point covariance function if $\beta < 1/2$, or for arbitrary $\beta \in (0,1)$ in the case of a general analytic covariance function. The following two classical results show that the requirement of analyticity can be relaxed to finite Sobolev regularity for any $\beta \in (0,1)$.

\begin{lemma}
\label{thm:bound_eigvals}
Let $C$ be symmetric and in $H^{k,0}(D^2) = H^{k}(D)\otimes H^0(D)$ for any $k \geq 0$. Denote by $(\lambda_m)_{m\geq 1}$ the sequence of eigenvalues of the covariance operator associated with $C$. Then it holds that
\[
0 \leq \lambda_m \lesssim m^{-k/d} \quad \forall m \geq 1.
\]
\end{lemma}
\begin{proof}
The proof can be found e.g. at \cite[Proposition 2.5]{frauenfelder_finite_2005}, or \cite[Theorem 3.1]{bachmayr_representations_2016} or \cite[Theorem 7.60]{lord_introduction_2014} in terms of the Mat\'ern smoothness parameter.
\end{proof}

\begin{lemma}
\label{thm:bound_eigfuns}
Let $C$ be symmetric, continuously differentiable and in $H^{3,0}(D^2)$. Denote by $(\lambda_m,b_m)_{m\geq 1}$ the sequence of eigenpairs of the covariance operator associated with $C$ such that $\|b_m\|_{L^2(D)}=1$ for all $m\geq 1$. Then for any $s > 0$ it holds that
\begin{align*}
\|b_m\|_{L^{\infty}(D)} &\lesssim |\lambda_m|^{-s},
\\
\|\nabla b_m\|_{L^{\infty}(D)} &\lesssim |\lambda_m|^{-s}.
\end{align*}
\end{lemma}
\begin{proof}
Because $C \in H^{3,0}(D^2)$, it follows from \cite[Proposition A.2]{schwab_karhunenloeve_2006} that $b_m \in H^3(D)$ for all $m \geq 1$. Let $\kappa$ be any multiindex such that $|\kappa| \in \{0,1\}$. Then it follows from the Ehrling-Nirenberg-Gagliardo inequality \cite[Lemma 3.1]{jiang_global_2014} that
\begin{align*}
\|\partial^{\kappa}b_m\|_{L^{\infty}(D)} &\lesssim \|\partial^{\kappa}b_m\|^{1/2}_{H^2(D)}\|\partial^{\kappa}b_m\|_{L^2(D)}^{1/2}
\\
&\lesssim \|\partial^{\kappa}b_m\|_{H^2(D)}
\\
&\lesssim \|b_m\|_{H^3(D)}.
\end{align*}
The result can then be obtained analogously to the proof of \cite[Theorem 2.24]{schwab_karhunenloeve_2006}.
\end{proof}
\begin{remark}
The above lemma is in fact a special case of \cite[Theorem 2.24]{schwab_karhunenloeve_2006}, where the same bound is obtained for arbitrarily high orders of derivatives of $b_m$ in the case $C \in \mathcal{C}^{\infty}$. When only bounds on $b_m$ and its gradient are required, the smoothness requirement can be relaxed to $C \in H^{3,0}$, which enables the use of Mat\'ern covariances with $\nu \geq 1$.
\end{remark}

\subsubsection{$L^p$-convergence for non-lognormal fields}

The analysis of \cite{charrier_strong_2012} shows that under assumptions $B1-B3$ a log-normal field based on a truncated KL expansion of a Gaussian field converges in $L^p$-sense to a log-normal field based on the true Gaussian field, and this analysis was extended by \cite{guth_generalized_2024} to a more general setting. The next proposition shows that it can additionally be extended to Lipschitz continuous transformations of Gaussian fields that are approximated by the $\varphi$-term shown in Equation \eqref{eq:transformation}.

\begin{proposition}
\label{thm:bound_afield}
Let $a$ be a Lipschitz continuous transformation of a Gaussian field $g$, with Lipschitz constant $K \geq 0$, and let $g_M$ be the approximation of $g$ obtained by truncating its KL expansion after $M$ terms. Denote
\begin{align*}
a_{\varphi}(x,\omega) &= \varphi\cdot\dfrac{\exp(g(x,\omega))}{1+\exp(g(x,\omega))} + a(g(x,\omega)),
\\
a_{\varphi,M}(x,\omega) &= \varphi\cdot\dfrac{\exp(g_M(x,\omega))}{1+\exp(g_M(x,\omega))} + a(g_M(x,\omega)).
\end{align*}
Assume further that assumptions $B1-B3$ are satisfied and denote
\[
R_M^{\beta} = \max\left(\sum_{m > M}\lambda_m\|b_m\|_{L^{\infty}(D)}^2, \ \sum_{m > M}\lambda_m\|b_m\|_{L^{\infty}(D)}^{2(1-\beta)}\|\nabla b_m\|_{L^{\infty}(D)}^{2\beta}\right).
\]
Then for any fixed $\omega$, fixed $\varphi, p > 0$ and $\beta \in (0,1)$, the following inequality holds for any $M \in \mathbb{N}$
\[
\|a_{\varphi,M} - a_{\varphi}\|_{L^p(\mathcal{C}^0(\overline{D}), \Omega)} \lesssim (R_M^{\beta})^{1/2},
\]
up to a constant that does not depend on $M$.
\end{proposition}
\begin{proof}
Denote by $\tilde{a},\tilde{a}_{\varphi}$ the same transformations as $a,a_{\varphi}$ respectively, acting on $\mathbb{R}$ instead of $g$. We first have the following inequality for all $x,y \in \mathbb{R}$, based on the Lipschitz continuity of $a$:
\begin{align*}
|\tilde{a}_{\varphi}(x) - \tilde{a}_{\varphi}(y)| &= \left\vert\varphi\dfrac{e^x}{1+e^x} + \tilde{a}(x) - \varphi\dfrac{e^y}{1+e^y} - \tilde{a}(y)\right\vert
\\
&\leq \varphi |e^x-e^y| + K|x-y|
\\
&\leq |x-y|\left(K + \varphi e^x + \varphi e^y\right).
\end{align*}
Now take any $p > 0$ and choose $q,r > 0$ such that $\frac{1}{p} = \frac{1}{q} + \frac{1}{r}$. It then follows from H\"older's inequality that
\[
\|a_{\varphi,M} - a_{\varphi}\|_{L^p(\mathcal{C}^0(\overline{D}), \Omega)} \lesssim \|g_M - g\|_{L^q(\mathcal{C}^0(\overline{D}), \Omega)}\left\| K + \varphi e^{g_M} + \varphi e^g \right\|_{L^r(\mathcal{C}^0(\overline{D}), \Omega)}.
\]
Because assumptions B1-B3 are satisfied, the following inequality holds up to a constant only depending on $\beta$ and $q$ \cite[Proposition 3.4]{charrier_strong_2012}
\[
\|g_M - g\|_{L^q(\mathcal{C}^0(\overline{D}), \Omega)} \lesssim (R_M^{\beta})^{1/2}.
\]
In addition, for every $r > 0$ there exists a constant $D_r$ only depending on $r$ such that \cite[Proposition 3.10]{charrier_strong_2012}
\[
\|e^{g_M}\|_{L^r(\mathcal{C}^0(\overline{D}),\Omega)} \leq D_r.
\]
Combining these expressions, we find that
\[
\|a_{\varphi,M}-a_{\varphi}\|_{L^p(\mathcal{C}^0(\overline{D}),\Omega)} \lesssim (R_M^{\beta})^{1/2}\left(K + \varphi D_r + \varphi \|e^g\|_{L^r(\mathcal{C}^0(\overline{D}),\Omega)}\right).
\]
Because $e^g$ is bounded in $L^p$-sense for all $p > 0$ \cite[Proposition 2.3]{charrier_strong_2012}, this concludes the proof by taking, for example, $r=q=2p$.
\end{proof}

\subsubsection{Strong convergence of PDE solution}

The following theorem shows the strong convergence of the PDE solution in terms of the truncation number of the KL expansion. We will only look at positive transformations $a$ and additionally require a weak upper bound on $a$. This theorem in particular motivates our choice of workaround: while the necessary bounds do hold for transformed distributions where $\varphi > 0$ (and that meet the assumptions of this theorem), they do not hold in the limit of $\varphi = 0$. However, the validity for any $\varphi > 0$ means that we can approximate the distributions, which suffices for many practical applications. In the context of this theorem, $u_M$ is the forward solution of the PDE based on a truncated KL expansion after $M$ terms of the Gaussian field $g$. The proof is based on the proof of \cite[Theorem 4.2]{charrier_strong_2012}, where similar results are derived for a diffusion equation. The estimates obtained there also hold for the linear elasticity equations \eqref{eq:lin_elastic}, as shown in \cite[Chapter 11]{brenner_mathematical_2010}.

\begin{theorem}
\label{thm:bound_pdesol}
Let $a$ be a positive, Lipschitz continuous transformation of a Gaussian field $g$ such that $\max_{x\in D}a(x,\omega) \leq \exp(\|g(\omega)\|_{\mathcal{C}^0(D)})$. If assumptions B1-B3 hold then for all $p>0$, $u_M$ converges to $u$ in $L^p(H_0^1(D),\Omega)$. Furthermore, for any $\delta > 0$,
\[
\|u-u_M\|_{L^p(H_0^1(\overline{D}),\Omega)} \lesssim M^{-\frac{1}{2}\left(\frac{k}{d}-1\right) + \delta}.
\]
\end{theorem}
\begin{proof}
Denote $a^{\min}_{\varphi} = \min_{x\in D}a_{\varphi}(g(x,\omega))$ and $a^{\min}_{M,\varphi} = \min_{x\in D}a_{M,\varphi}(g_M(x,\omega))$, where $a_{\varphi}, a_{M,\varphi}$ are defined as in Proposition \ref{thm:bound_afield}. We then have the following inequality
\begin{align*}
a^{\min}_{\varphi}(\omega) \geq \dfrac{\varphi}{2}\min\left\lbrace 1, e^{-\|g(\omega)\|_{\mathcal{C}^0(D)}} \right\rbrace.
\end{align*}
It follows from this inequality, together with $\max_{x\in D}a(x,\omega) \leq \exp(\|g(\omega)\|_{\mathcal{C}^0(D)})$ and \cite[Proposition 2.3]{charrier_strong_2012} that $\frac{1}{a_{\varphi}^{\min}} \in L^p(\Omega)$ for all $p > 0$. Similarly, these inequalities together with \cite[Proposition 3.10]{charrier_strong_2012} imply that $\frac{1}{a^{\min}_{M,\varphi}} \in L^p(\Omega)$ for all $p > 0$. 

Now choose $q,r,s > 0$ such that $\frac{1}{p} = \frac{1}{q} + \frac{1}{r} + \frac{1}{s}$. Then it follows from H\"older's inequality and \cite[Proposition 4.1]{charrier_strong_2012} that
\[
\|u-u_M\|_{L^p(H_0^1(\overline{D}),\Omega)} \leq \left\|\dfrac{1}{a^{\min}_{M,\varphi}}\right\|_{L^q(\Omega)}\|a_{\varphi}-a_{M,\varphi}\|_{L^r(\mathcal{C}^0(\overline{D}),\Omega)}\|b\|_{L^2}C_D\left\|\dfrac{1}{a_{\varphi}^{\min}}\right\|_{L^s(\Omega)},
\]
with $C_D$ a constant depending only on the domain $D$. It then follows from Proposition \ref{thm:bound_afield} that
\[
\|u-u_M\|_{L^p(H_0^1(\overline{D}),\Omega)} \lesssim (R_M^{\beta})^{1/2}.
\]
We will now calculate a bound on $R_M^{\beta}$. By Lemmas \ref{thm:bound_eigvals} and \ref{thm:bound_eigfuns} it holds for any $s > 0$ that
\begin{align*}
R_M^{\beta} \lesssim \sum_{m>M}|\lambda_m|^{1-2s} \lesssim \sum_{m>M} m^{-k/d(1-2s)} \lesssim M^{1-(1-2s)\frac{k}{d}}.
\end{align*}
The theorem then follows by taking $s = d\delta/k$.
\end{proof}

\subsection{Cost of multilevel MCMC estimator}
\label{seq:analysis_data}

In this next subsection, we will turn our attention to the multilevel MCMC algorithm. Using the results from Subsection \ref{seq:analysis_field}, we will adapt the convergence analysis of \cite{dodwell_hierarchical_2015} to obtain an upper bound on the $\varepsilon$-cost of the multilevel MCMC estimator in the case of Mat\'ern covariance. Furthermore, we show how the level-dependent treatment of observations can be integrated in the existing convergence analysis. Similarly to before, we will not repeat the entire convergence analysis, but only show where alterations are needed such that the results of \cite{dodwell_hierarchical_2015} hold in our case.

\subsubsection{Main theorem: cost of multilevel estimator}

The main result is Theorem \ref{thm:epscost}, giving a bound on the $\varepsilon$-cost (i.e. the cost of obtaining a root mean squared error $<\varepsilon$) of the multilevel estimator under a set of assumptions. Referring to the notation introduced in Subsection \ref{subsec:underlying}, in the context of this theorem $\pmb{\Theta}_{\ell} = \{\theta_{\ell}^n\}_{n\in\mathbb{N}} \cup \{\Theta_{\ell-1}^n\}_{n\in\mathbb{N}}$ for $\ell \geq 1$ and $\pmb{\Theta}_0 = \{\theta_0^n\}_{n\in\mathbb{N}}$. $\nu^{\ell,\ell-1}$ is the joint distribution of $\theta_{\ell}$ and $\Theta_{\ell-1}$ for $\ell\geq 1$, where the marginals of $\theta_{\ell}$ and $\Theta_{\ell-1}$ are $\nu^{\ell}$ and $\nu^{\ell-1}$, respectively. $\rho$ is the true distribution of $Q(\vartheta)$, and $\mathbb{E}_{\bullet}$ and $\mathbb{V}_{\bullet}$ denote the expected value and variance of a quantity under the distribution in their subscript, respectively. $\mathcal{C}_{\ell}$ denotes the cost of evaluating the forward model at level $\ell$. For any subscript, $C_{\bullet}$ denotes a real-valued number independent of $\ell$. Finally, for consistency we let $Y_0 = Q_0$, $\nu^{0,-1}=\nu^0$ and $R_{-1}=1$.
 For more details and a general proof of this theorem, we refer to \cite{dodwell_hierarchical_2015}.

\begin{theorem}
\label{thm:epscost}
Let $\varepsilon < \exp(-1)$, and suppose there are positive constants $\alpha, \alpha', \beta, \beta',$ $ \gamma > 0$ such that $\alpha \geq \frac{1}{2} \min(\beta, \gamma)$. Under the following assumptions, for $\ell = 0,\ldots,L$,
\begin{enumerate}[label=M\arabic*.]
\item $|\mathbb{E}_{\nu^{\ell}}[Q_{\ell}] - \mathbb{E}_{\rho}[Q]| \leq C_{\text{M1}}\ (R_{\ell}^{-\alpha} + M_{\ell}^{-\alpha'}),$
\item $\mathbb{V}_{\nu^{\ell,\ell-1}}[Y_{\ell}] \leq C_{\text{M2}}\ (R_{\ell-1}^{-\beta} + M_{\ell-1}^{-\beta'}),$
\item $\mathbb{V}_{\pmb{\Theta}_{\ell}}[\widehat{Y}_{\ell, N_{\ell}}^{MC}] + (\mathbb{E}_{\pmb{\Theta}_{\ell}}[\widehat{Y}_{\ell, N_{\ell}}^{MC}] - \mathbb{E}_{\nu^{\ell,\ell-1}}[\widehat{Y}_{\ell, N_{\ell}}^{MC}])^2 \leq C_{\text{M3}}\ N_{\ell}^{-1}\mathbb{V}_{\nu^{\ell,\ell-1}}[Y_{\ell}]$,
\item $\mathcal{C}_{\ell} \leq C_{\text{M4}}\ R_{\ell}^{\gamma}$,
\end{enumerate} 
and, provided that $M_{\ell} \gtrsim R_{\ell}^{\max\{\alpha/\alpha', \beta/\beta'\}}$, there exist a number of levels $L$ and a sequence $\{n_{\ell}\}_{\ell=0}^L$ such that
\[
e(\widehat{Q}_{L,\{n_{\ell}\}}^{ML})^2 := \mathbb{E}_{\cup_{\ell}\Theta_{\ell}}\left[\left(\widehat{Q}_{L,\{n_{\ell}\}}^{ML} - \mathbb{E}_{\rho}[Q]\right)^2\right] < \varepsilon^2
\]
and
\[
\mathcal{C}_{\varepsilon}(\widehat{Q}_{L,\{N_{\ell}\}}^{ML}) \leq C_{\text{ML}} \left\{
\begin{array}{ll}
\varepsilon^{-2}|\log\varepsilon|\quad & \text{if }\ \beta > \gamma,
\\
\varepsilon^{-2}|\log\varepsilon|^3\quad & \text{if }\ \beta = \gamma,
\\
\varepsilon^{-2-(\gamma-\beta)/\alpha}|\log\varepsilon|\quad & \text{if }\ \beta < \gamma.
\end{array}
\right.
\]
\end{theorem}

Assumption M3 involves bounding the mean squared error of an MCMC estimator, and is satisfied provided the chains are sufficiently burnt in. Assumption M4 constitutes an upper bound on the cost of a single sample on level $\ell$, and in a best case scenario this is met for $\gamma = 1$. For proofs of the general validity of these two assumptions we refer to \cite{dodwell_hierarchical_2015}.

\subsubsection{Proof of validity of assumptions} 

Assumptions M1-M2 require some additional attention. In order to prove them, we need the results from the previous subsection, as well as the following three assumptions:

\begin{enumerate}[label=A\arabic*.]
\item $\mathcal{F}: H^1(D) \to \mathbb{R}$ is linear, and there exists $C_{\mathcal{F}} \in \mathbb{R}$ such that
\[
|\mathcal{F}(v)| \leq C_{\mathcal{F}}\|v\|_{H^{1/2-\delta}} \quad \text{for all }\delta > 0.
\]
\item The sequence of fidelity parameters $\{\sigma
^2_{F,\ell}\}_{\ell=0}^{\infty}$ satisfies
\[
|N_L^{-1}\sigma_F^{-2} - N_{\ell}^{-1}\sigma_{F,\ell}^{-2}| \lesssim \max\left(M_{\ell}^{-\frac{1}{2}(\frac{k}{d}-1)+\delta}, R_{\ell}^{-1/d+\delta}\right)\quad \text{for all }\delta > 0.
\]
\item There exists $C_W \in \mathbb{R}$ such that the sequence of level-dependent observation functions $\{W_{\ell}:\mathbb{R}^{N_L}\to \mathbb{R}^{N_{\ell}}\}_{\ell=0}^{L}$ satisfies
\[
\|W_{\ell}(u^{\text{obs}}) - W^S_{\ell}(u^{\text{obs}})\| \leq C_W \max\left(M_{\ell}^{-\frac{1}{2}(\frac{k}{d}-1)+\delta}, R_{\ell}^{-1/d+\delta}\right)\quad \text{for all }\delta > 0.
\]
\end{enumerate}

Assumption A1 is the same as in the original convergence analysis \cite{dodwell_hierarchical_2015}. Assumption A2 now includes the number of data points $N_{\ell}$ used on each level $\ell$. As $N_{\ell} = N_L$ for $\ell \geq L$, it holds in the same cases as the original level-independent version of this assumption (where $N_{\ell}=N_L$ for all $\ell$). Assumption A3 allows for treatments of observations on coarse levels other than simply selecting only those coinciding with nodes on the coarser level, as long as these more general weight functions do not deviate too strongly from $W_{\ell}^S$. Using these three assumptions we can first prove the validity of assumption M1.

\begin{proposition}
\label{thm:M1holds}
Let $g$ be a Gaussian field with constant mean and covariance function in $H^{k,0}(D^2)$ with $k\geq 3$, and let $a$ be a positive, Lipschitz continuous transformation of $g$ such that $\max_{x\in D}a(x,\omega) \leq \exp(\|g(\omega)\|_{\mathcal{C}^0(D)})$. Suppose $D \subset \mathbb{R}^d$ is Lipschitz polygonal and that assumptions A1-A3 hold. Then assumption M1 is satisfied with $\alpha = 1/d - \delta$ and $\alpha' = \frac{1}{2}(\frac{k}{d}-1) - \delta$ for any $\delta > 0$.
\end{proposition}
\begin{proof}
The proof follows closely the proof of \cite[Lemma 4.5]{dodwell_hierarchical_2015}. It follows from this lemma and theorem \ref{thm:bound_pdesol} that
\begin{align*}
|\mathbb{E}_{\nu^{\ell}}[Q_{\ell}] - \mathbb{E}_{\rho}[Q]| &\leq C_Q \left(R_{\ell}^{-1/d+\delta} + M_{\ell}^{-\frac{1}{2}(\frac{k}{d}-1)+\delta}\right) + d_{\text{Hell}}(\rho,\rho^{\ell}).
\end{align*}
For the second term on the right hand side, we introduce $\rho^{\ell}$ as the posterior measure of $\vartheta$ with the likelihood evaluated at level $\ell$. This gives us, using assumption A1, a bound on the Hellinger distance between $\rho$ and $\rho^{\ell}$:
\begin{align*}
d_{\text{Hell}}^2(\rho,\rho^{\ell}) \leq \mathbb{E}_{\rho_0}\left[\left\vert\dfrac{\|u^{\text{obs}} - F\|^2}{N_L\sigma_F^2} - \dfrac{\|W_{\ell}(u^{\text{obs}}) - F_{\ell}\|^2}{N_{\ell}\sigma_{F, \ell}^2}\right\vert^2\right].
\end{align*}

The expression inside the expectation can be bounded by use of the triangle inequality
\begin{align*}
&\left\vert\dfrac{\|u^{\text{obs}} - F\|^2}{N_L\sigma_F^2} - \dfrac{\|W_{\ell}(u^{\text{obs}}) - F_{\ell}\|^2}{N_{\ell}\sigma_{F, \ell}^2}\right\vert
\\
&\leq  \left\vert\dfrac{\left(\|W_{\ell}^S(u^{\text{obs}}) - W_{\ell}(u^{\text{obs}})\| + \|W_{\ell}(u^{\text{obs}})-F_{\ell}\| + \|W_{\ell}^S(F)-F_{\ell}\|\right)^2}{N_L\sigma_F^2} - \dfrac{\|W_{\ell}(u^{\text{obs}}) - F_{\ell}\|^2}{N_{\ell}\sigma_{F, \ell}^2}\right\vert
\\
&= \dfrac{2\|W_{\ell}(u^{\text{obs}}) - F_{\ell}\|+2\|W_{\ell}^S(F)-F_{\ell}\|+\|W_{\ell}^S(u^{\text{obs}}) - W_{\ell}(u^{\text{obs}})\|}{N_L\sigma_F^2}\|W_{\ell}^S(u^{\text{obs}}) - W_{\ell}(u^{\text{obs}})\|
\\
&\hspace*{10mm}+ \dfrac{2\|W_{\ell}(u^{\text{obs}}) - F_{\ell}\|+\|W_{\ell}^S(F)-F_{\ell}\|}{N_L\sigma_F^2}\|W_{\ell}^S(F)-F_{\ell}\|
\\
&\hspace*{20mm} + \|W_{\ell}(u^{\text{obs}}) - F_{\ell}\|^2 \left\vert N_L^{-1}\sigma_F^{-2} - N_{\ell}^{-1}\sigma_{F,\ell}^{-2}\right\vert.
\end{align*}
From assumption A1, theorem \ref{thm:bound_pdesol} and \cite[Corollary 4.3]{dodwell_hierarchical_2015} it holds for all $q<\infty$ that
\[
\mathbb{E}_{\rho_0}\left[\|W_{\ell}^S(F) - F_{\ell}\|^q\right]^{1/q} \leq C_{\text{A1}} \left(R_{\ell}^{-1/d+\delta} + M_{\ell}^{-\frac{1}{2}(\frac{k}{d}-1)+\delta}\right).
\]
Combining this with assumptions A2-A3, it follows that 
\[
\mathbb{E}_{\rho_0}\left[\left\vert\dfrac{\|u^{\text{obs}} - F\|^2}{N_L\sigma_F^2} - \dfrac{\|W_{\ell}(u^{\text{obs}}) - F_{\ell}\|^2}{N_{\ell}\sigma_{F, \ell}^2}\right\vert^2\right] \leq C_{\text{Hell}} \left(R_{\ell}^{-1/d+\delta} + M_{\ell}^{-\frac{1}{2}(\frac{k}{d}-1)+\delta}\right)^2,
\]
which concludes the proof.
\end{proof}

A similar bound on assumption M2 can be proven using the results we obtained previously.

\begin{proposition}
\label{thm:M2holds}
Let $\theta_{\ell}^{n+1}$ and $\Theta_{\ell-1}^{n+1}$ have joint distribution $\nu^{\ell,\ell-1}$ and set $Y_{\ell}^{n+1} = Q_{\ell}(\theta_{\ell}^{n+1})$ $- Q_{\ell-1}(\Theta_{\ell-1}^{n+1})$. If $q_{ML}^{\ell,F}$ is a pCN proposal distribution, then assumption M2 is satisfied with $\beta = 1/d - \delta$ and $\beta' = \frac{1}{2}(\frac{k}{d}-1) - \delta$ for any $\delta > 0$.
\end{proposition}
\begin{proof}
This follows immediately from \cite[Lemmas 4.6-4.8]{dodwell_hierarchical_2015}, Theorem \ref{thm:bound_pdesol}, and the triangle inequality estimate used in the proof of Proposition \ref{thm:M1holds}.
\end{proof}

The results of these propositions show that the assumptions of theorem \ref{thm:epscost} hold in the same way for level-dependent treatment as they did for level-independent treatment \cite{dodwell_hierarchical_2015}. Additionally, we obtain different values for the parameters $\alpha',\beta'$ for the case of finite Sobolev regularity of the covariance function. In the case of exponential two-point covariance, the assumptions hold with $\alpha'=\beta'=\frac{1}{2}-\delta$ for any $\delta > 0$. For larger values of $k$, we find larger values of $\alpha'$ and $\beta'$, leading to faster convergence of the terms in assumptions M1 and M2.

\section{Numerical experiments}
\label{ch:numerics}

In this section, we will run the multilevel MCMC algorithm with level-dependent data treatment for the clamped beam model outlined in Section \ref{ch:beammodel}. We consider a beam with a length of 3 m and a height of 0.3 m. We further assume that the material is linearly elastic and isotropic with the following values as for concrete: a Poisson's ratio of 0.25 and a material density of 2500 kg/m$^3$. The variable Young's modulus will be normalised with respect to a value of 30 GPa. The line load is assumed to exert a total downward force of $10^3$ N across its length. Each experiment consists of 35 parallel runs using a coarsest chain length of $1.5\cdot 10^7$ samples, resulting in slightly over $5\cdot 10^8$ total coarsest-level samples.

We consider synthetic displacement observations, generated from an artificial ground truth for the stiffness field constructed on the finest level of the multilevel hierarchy. This field is obtained by sampling from a $\mathcal{N}(0, I_{M_L})$ prior distribution, with $M_L$ the truncation number of the KL expansion at the finest level. We assume to have a measurement error $\eta$ corresponding to a fidelity of $\sigma_F = 10^{-8}$ m, which corresponds to an order of magnitude of about one percent on the observed displacement, which typically on the order of $10^{-6}$ m. The samples are first fed through the KL expansion in Equation \eqref{eq:normalfield} and then the transformation of Equation \eqref{eq:transformation} with scale and shape parameters $\mu=0.4$ and $\kappa=2.5$, respectively, to obtain the desired properties for the spatially varying Young's modulus.

We will support our theoretical work with two experiments. First, in Subsection \ref{sec:exp_smoothness}, we will verify the values for $\alpha'$ and $\beta'$ found in Propositions \ref{thm:M1holds} and \ref{thm:M2holds} for the case of a Mat\'ern covariance function. Second, in Subsection \ref{sec:exp_singlemulti}, we illustrate the gains of the multilevel method over classical single-level MCMC, paying special attention to the cost of evaluating high-resolution data. Additionally, Subsection \ref{sec:reconstruction} shows the use of the method to obtain posterior estimates.

All finite element computations shown in this section are performed using the Stabil toolbox for Matlab, developed at KU Leuven \cite{francois_stabil_2021}. The code used to generate the results is open-source, and can be accessed at \href{https://gitlab.kuleuven.be/u0114679/mlmcmc-beam-model}{https://gitlab.kuleuven.be/u0114679/mlmcmc-beam-model}.

\subsection{Effect of covariance smoothness}
\label{sec:exp_smoothness}

The convergence analysis of Section \ref{ch:conv_analysis} shows that, for a general covariance function of finite Sobolev regularity $k$, assumptions M1 and M2 with $\alpha'=\beta'=\frac{1}{2}(\frac{k}{d}-1)-\delta$ for all $\delta >0$. We will now conduct a simple numerical experiment to verify this claim.

\subsubsection{Experiment setup}

We run a 5-level multilevel MCMC routine for the linear elastic beam model, using as quantity of interest $Q$ the spatial average of the stiffness value within the bottom center third of the beam. The parameters used for this run can be found in Table \ref{tab:exp1}. The level-dependent fidelity $\sigma_{F,\ell}$ is taken equal to $\sigma_F$ at each level. In this experiment, we take as truncation number $M_{\ell}=50\cdot(\ell+1)$. To show the dependency of the convergence on $M_{\ell}$, we do not vary $R_{\ell}$ in this case, but keep it constant at an intermediately coarse discretisation of 30 elements in horizontal direction and 24 in vertical, about the highest discretisation level feasible for this experiment, which does not exploit different resolutions. This is of course not a realistic setup in practice as the multilevel method is primarily designed to be used for high-resolution setups, but it allows us to check the predicted values of $\alpha'$ and $\beta'$, without having any contribution from the terms of $\alpha$ and $\beta$.

\begin{table}[h]
\centering
\begin{tabular}{|c|c|c|c|c|c|c|} 
\hline
$\ell$ & 0 & 1 & 2 & 3 & 4 \\
\hline
$M_{\ell}$ & 50 & 100 & 150 & 200 & 250 \\
\hline
$R_{\ell}$ & 30$\times$24 & 30$\times$24 & 30$\times$24 & 30$\times$24 & 30$\times$24 \\
\hline
$\tau_{\ell}$ & & 100 & 5 & 5 & 5 \\
\hline
\end{tabular}
\caption{Parameters of the multilevel MCMC run used in the experiment of Subsection \ref{sec:exp_smoothness}.}
\label{tab:exp1}
\end{table}

Finally, our method requires a value of the subsampling rate $\tau_{\ell}$ between levels. Here we choose $\tau_1=100$ and a considerably lower value of $\tau_{\ell}=5$ for the higher levels. These values are chosen ad hoc because they provide a good balance between mixing and computational cost. A more theoretically substantiated choice would be to take $\tau_{\ell}$ equal to twice the integrated autocorrelation time (IAT), which can be estimated online. It is believed, however, that twice the IAT is an unnecessarily large value for the subsampling rate in practice, and that far smaller values can still provide good mixing for far less computational cost \cite{dodwell_hierarchical_2015}. This can be corroborated by our findings, as our method produces a highly autocorrelated chain at the coarsest level, with an IAT of approximately 2900, yet a subsampling rate of 100 seems sufficient to obtain well-mixed samples at the finer levels, and allows considerably more samples to be generated. Although a more in-depth cost-effect analysis of the subsampling rate would be of value, this is beyond the scope of this article.

\begin{figure}[h]
\centering
\begin{subfigure}[b]{0.5\textwidth}
\includegraphics[width=\textwidth]{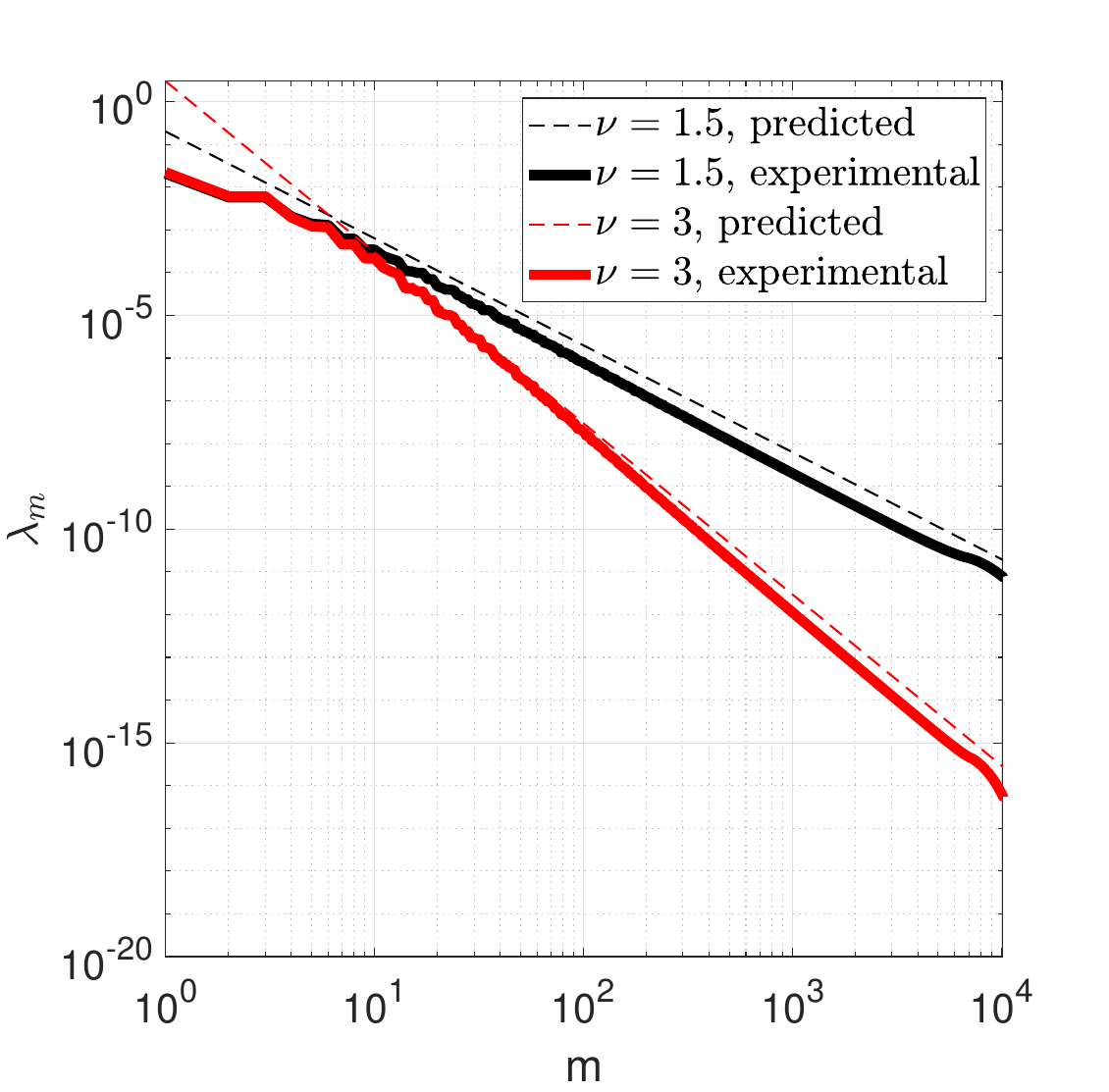}
\end{subfigure}%
\begin{subfigure}[b]{0.5\textwidth}
\includegraphics[width=\textwidth]{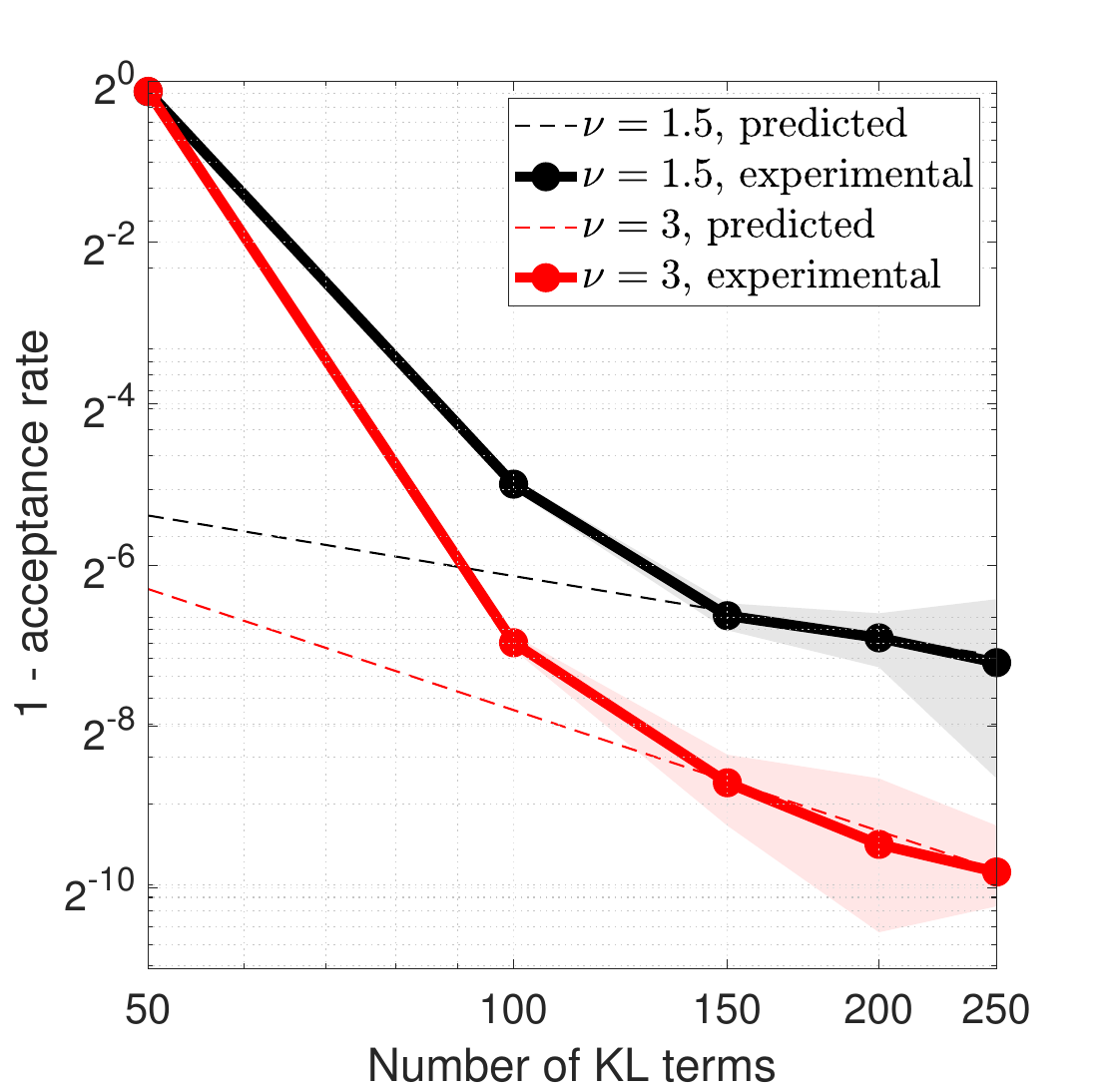}
\end{subfigure}
\caption{Left: eigenvalues of the Mat\'ern covariance operator for two different values of the smoothness parameter $\nu$. Right: corresponding convergence of the acceptance rate for the samples in the multilevel acceptance rate. In both figures, dashed lines are the theoretical trends, defined up to a constant, while solid lines are experimental results. The shaded regions on the right illustrate confidence intervals obtained by calculating the estimate on each of the 35 parallel chains separately.}
\label{fig:matern_exp}
\end{figure}

\subsubsection{Experiment results}

We run the same experiment for two different Mat\'ern covariance functions, both taken with values of $\lambda = 0.5,\ \sigma = 2$. The smoothness parameter is taken to be $\nu=1.5$ for the first experiment and $\nu=3$ for the second, corresponding to $k=5$ and $k=8$, respectively. The left plot of Figure \ref{fig:matern_exp} shows in dashed lines the expected eigenvalue decay for a covariance function as obtained from the bound in Lemma \ref{thm:bound_eigvals}. The solid lines show the numerically calculated eigenvalues for the choices of $\nu$.

The right plot of Figure \ref{fig:matern_exp} shows the level-dependent rejection rate of samples in the multilevel MCMC routine. From the proof of Proposition \ref{thm:M2holds} it follows that the specific value of $\beta'$ arises from the following bound on the rejection rate

\begin{equation}
\mathbb{E}_{\zeta}\left[\left(1-\alpha_{M_{\ell},R_{\ell}}(\theta_{\ell}'|\theta_{\ell}^n)\right)\right] \lesssim \left(R_{\ell}^{-1/d+\delta}+M_{\ell}^{-\frac{1}{2}(\frac{k}{d}-1)+\delta}\right),
\end{equation}

where $\zeta$ is the joint distribution of $\theta_{\ell}^n,\theta_{\ell}'$. This bound is shown by dashed lines. The empirical rejection rates at each level are shown in connected dots. We see a clear agreement between the theoretical predictions and the results of the experiment. The large rejection rate at the coarsest level is mainly a pre-asymptotic effect as all rates are obtained with respect to $\ell \to \infty$.

\subsection{Comparison between single- and multilevel MCMC}
\label{sec:exp_singlemulti}

We next conduct an experiment where both $M_{\ell}$ and $R_{\ell}$ are varied across the different levels in the multilevel algorithm. We now run a 6-level version of the method, with the same subsampling rates as above. We take as truncation number $M_{\ell}=50\cdot(\ell+1)$. The resolution $R_{\ell}$ of the FE grid is now taken as $15\times 12$ on the coarsest grid, and this time the number of elements in each direction is doubled at every level. This constitutes an approximate 1000-fold increase in elements between the coarsest and finest levels. The values for this experiment are summarized in Table \ref{tab:exp2}.

\begin{table}[h]
\centering
\begin{tabular}{|c|c|c|c|c|c|c|} 
\hline
$\ell$ & 0 & 1 & 2 & 3 & 4 & 5 \\
\hline
$M_{\ell}$ & 50 & 100 & 150 & 200 & 250 & 300 \\
\hline
$R_{\ell}$ & 15$\times$12 & 30$\times$24 & 60$\times$48 & 120$\times$96 & 240$\times$192 & 480$\times$384 \\
\hline
$\tau_{\ell}$ & & 100 & 5 & 5 & 5 & 5 \\
\hline
\end{tabular}
\caption{Parameters of the multilevel MCMC run used in the experiment of Subsection \ref{sec:exp_singlemulti}.}
\label{tab:exp2}
\end{table}

The left plot of Figure \ref{fig:ml_exp} shows the cost in CPU time of generating an individual sample at each level in this setup. The total cost has two main contributions: solving the linear system to obtain $F_{\ell}(\theta_{\ell})$, and evaluating the likelihood of Equation \eqref{eq:levellike}. The black line is the reference cost of a single-level method, where both parts are performed at the finest, most expensive level for each sample. In red is the cost for a direct application of the multilevel MCMC algorithm without the level-dependent treatment of data outlined in Subsection \ref{sec:leveldata}. Instead, the output $F_{\ell}(\theta_{\ell})$ is interpolated onto the finest-level mesh to be directly compared to $u^{\text{obs}}$. In this case, there is a clear reduction in sample cost at the middle levels $\ell=3,4$, because at these levels solving the forward model is the largest computational bottleneck. However, the cost reduction plateaus at the coarsest levels, because here the cost of solving the forward model drops below that of evaluating the likelihood function, which becomes the new bottleneck. This loss of efficiency is countered by the level-dependent data treatment shown in blue, where the lower cost is continued for the coarsest levels as well.

\begin{figure}[h]
\centering
\begin{subfigure}[b]{0.5\textwidth}
\includegraphics[width=\textwidth]{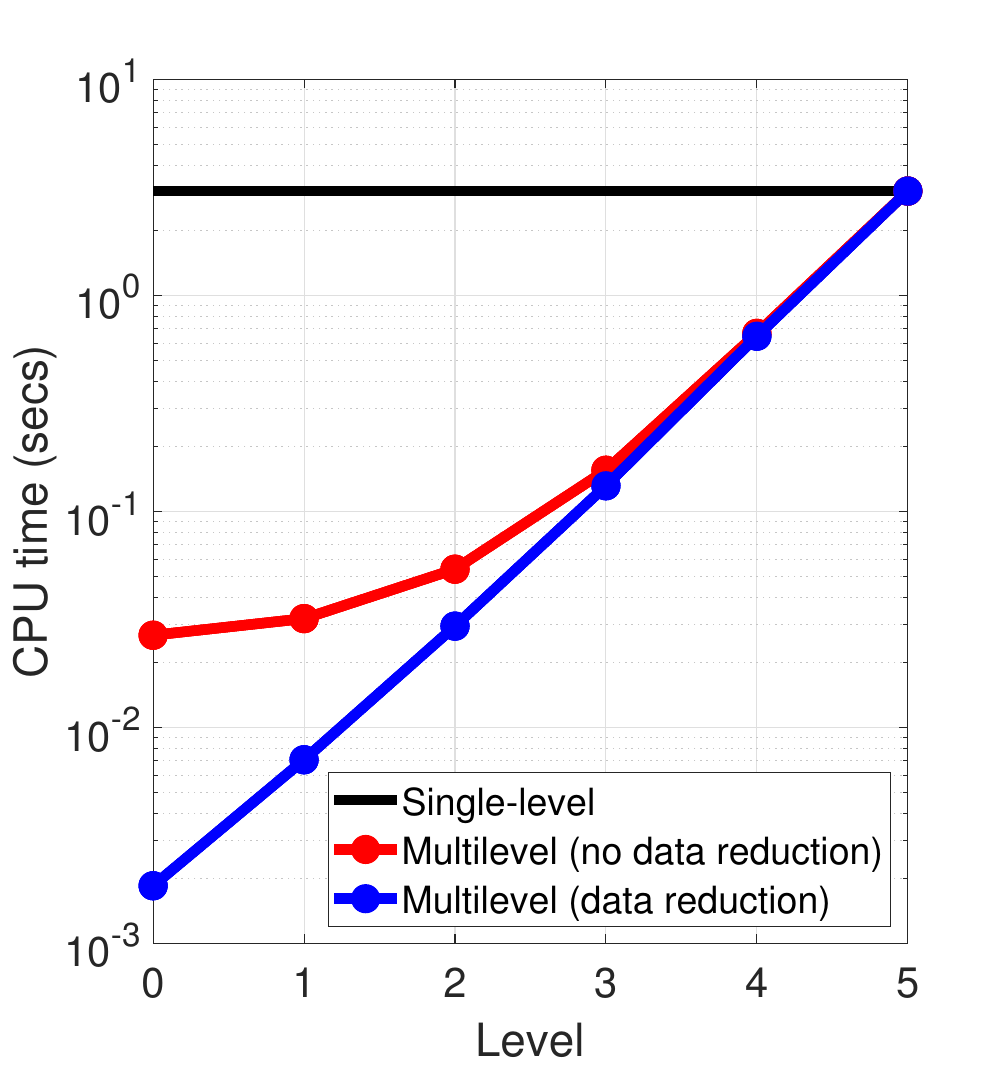}
\end{subfigure}%
\begin{subfigure}[b]{0.5\textwidth}
\includegraphics[width=\textwidth]{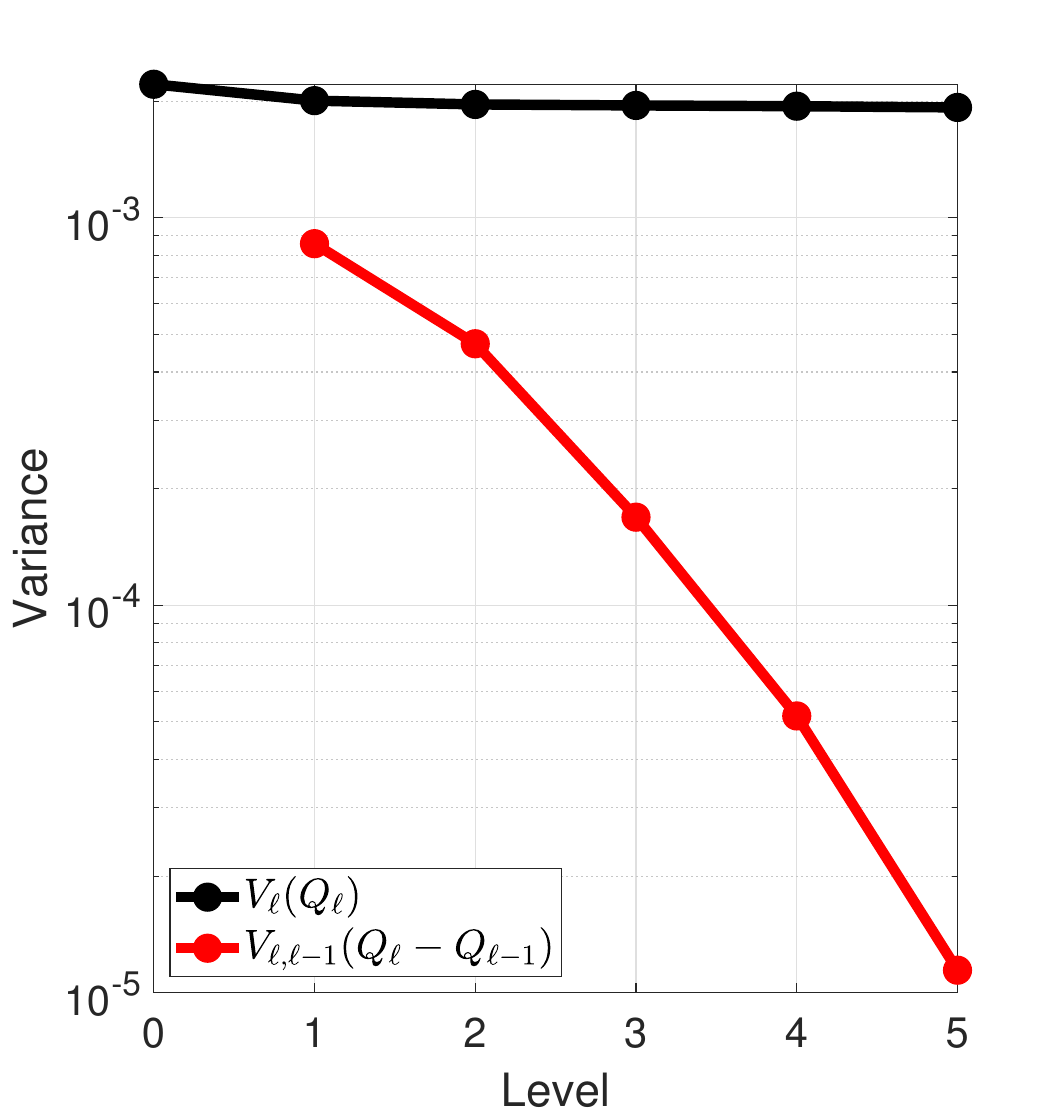}
\end{subfigure}
\caption{Left: cost of generating a single sample at each level of the multilevel algorithm. Right: variance of the single-level estimator for the quantity of interest at each level versus the correction terms in the multilevel approach.}
\label{fig:ml_exp}
\end{figure}

The right plot of Figure \ref{fig:ml_exp} shows the gain of using the multilevel correction terms $Q_{\ell}-Q_{\ell-1}$ over the classical approach of calculating the quantity of interest on one (high-resolution) level. As outlined by Theorem \ref{thm:epscost}, while the sample variance of the quantity of interest is more or less constant at each level (but with a higher Monte Carlo error due to a lower sample size because of the increased cost), the variances of the correction terms are much smaller. This well-known result for multilevel Monte Carlo methods \cite{dodwell_hierarchical_2015,giles_multilevel_2015} is thus also obtained by our method.

\subsection{Ground truth reconstruction}
\label{sec:reconstruction}

As a final experiment, we show the posterior estimates of a given stiffness field using the multilevel MCMC algorithm. The ground truth is generated by sampling 300 standard normal realisations for the $\{\xi_m\}$ which are fed first through the KL expansion and then the transformation of Equation \eqref{eq:transformation}. This ground truth is shown in the left plot of Figure \ref{fig:postmean}. We run the multilevel MCMC routine with the same values of $R_{\ell},\tau_{\ell}$ as in Section \ref{sec:exp_singlemulti}, and this time $M_{\ell} = 20\cdot (\ell+1)$. The resulting posterior mean for this field in shown in the right plot of Figure \ref{fig:postmean}.

\begin{figure}[h]
\centering
\begin{subfigure}[b]{0.9\textwidth}
\includegraphics[width=\textwidth]{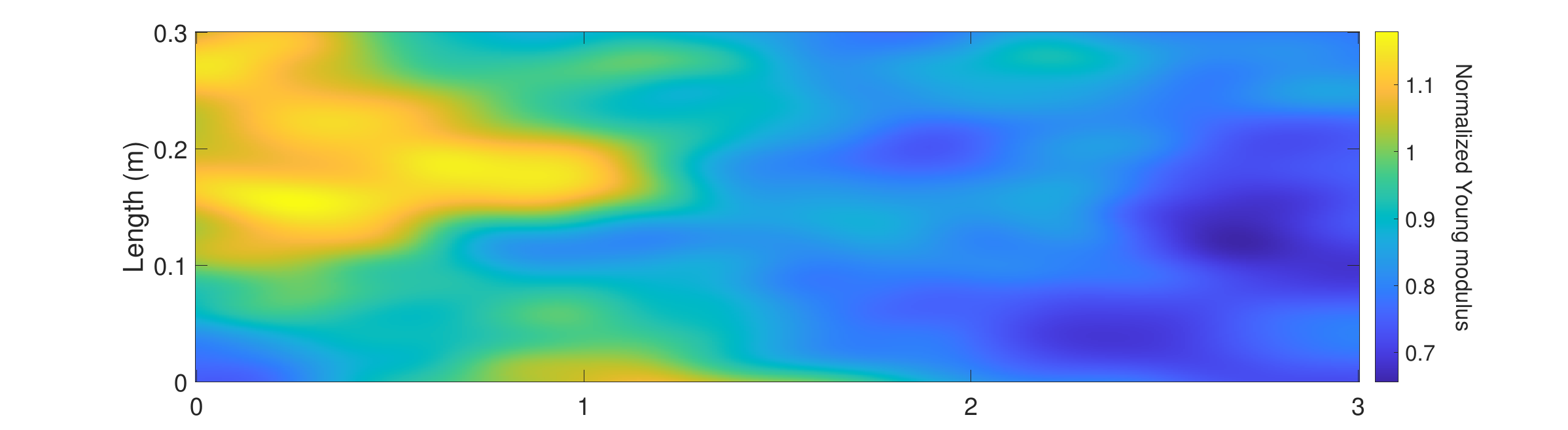}
\end{subfigure}
\begin{subfigure}[b]{0.9\textwidth}
\includegraphics[width=\textwidth]{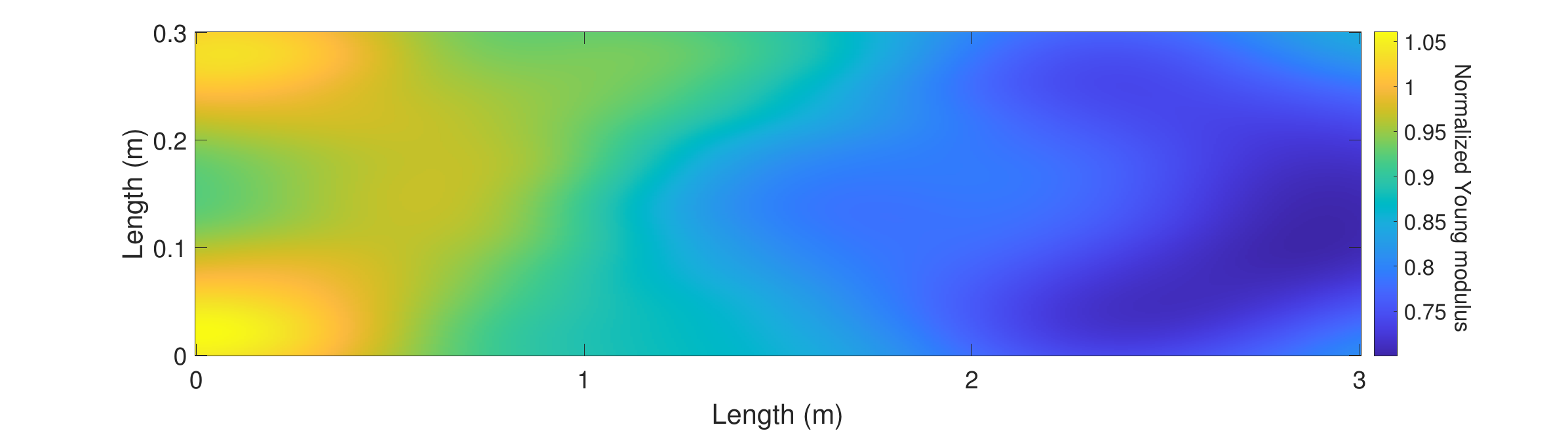}
\end{subfigure}
\caption{Top: Ground truth value of normalized Young's modulus. Bottom: posterior mean of multilevel MCMC estimator obtained at the finest level.}
\label{fig:postmean}
\end{figure}

The posterior mean, while smoothed out, is in good agreement overall with the ground truth, but of course only shows part of the picture. The multilevel MCMC method allows us to also obtain samples at the finest level that are distributed according to the posterior at the finest level. There are many ways of employing these samples to obtain uncertainty estimates, but the simplest method that is commonly used in structural mechanics contexts is simply looking at individual samples and comparing their likelihood (and prior) scores \cite{tarantola_inverse_2005}. Figure \ref{fig:samples} contains 4 such samples, which shows that there is quite a lot of variability in samples that nevertheless all have similar likelihood scores. This is important information for practitioners, as it shows that the posterior mean should not be carelessly used, highlighting one of the benefits of (multilevel) MCMC approaches.

\begin{figure}[h!]
\centering
\begin{subfigure}[b]{0.9\textwidth}
\includegraphics[width=\textwidth]{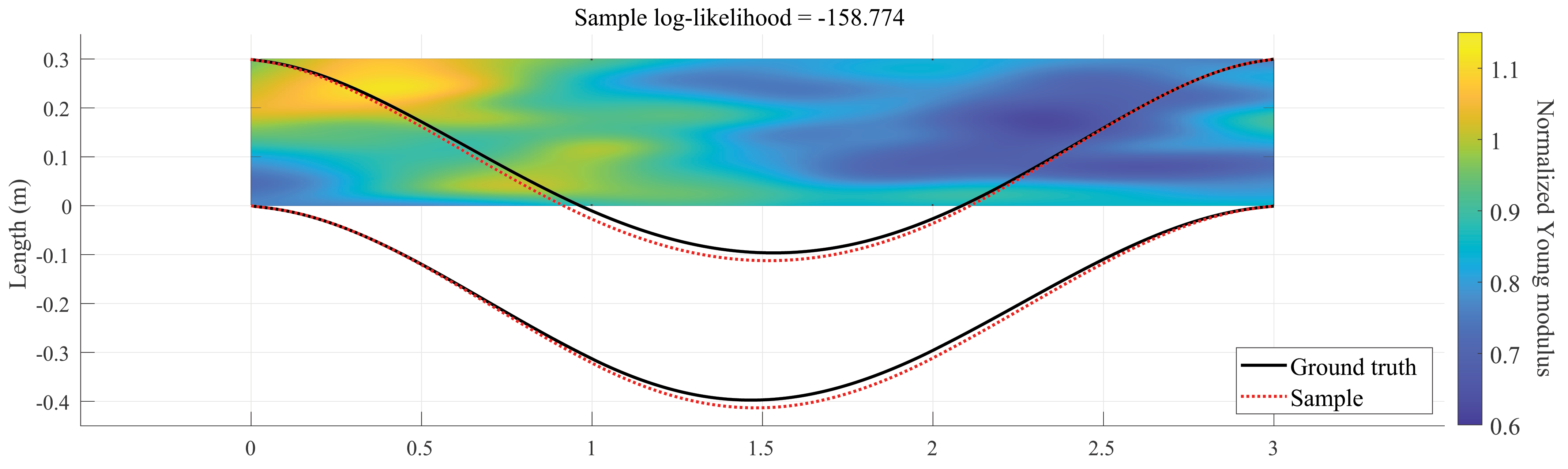}
\end{subfigure}
\begin{subfigure}[b]{0.9\textwidth}
\includegraphics[width=\textwidth]{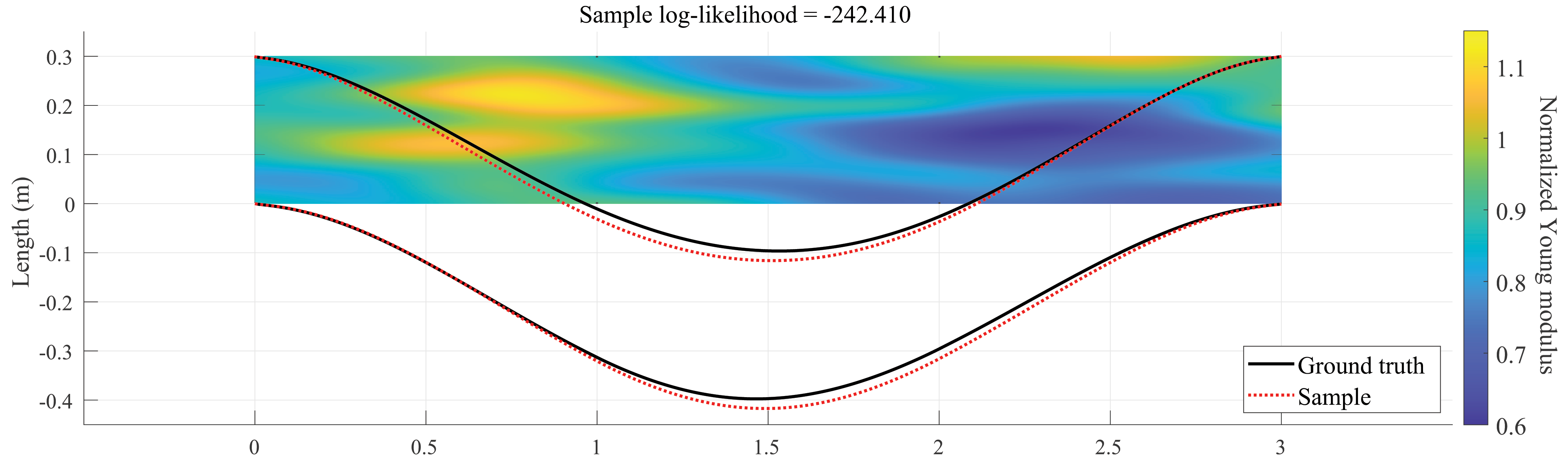}
\end{subfigure}
\begin{subfigure}[b]{0.9\textwidth}
\includegraphics[width=\textwidth]{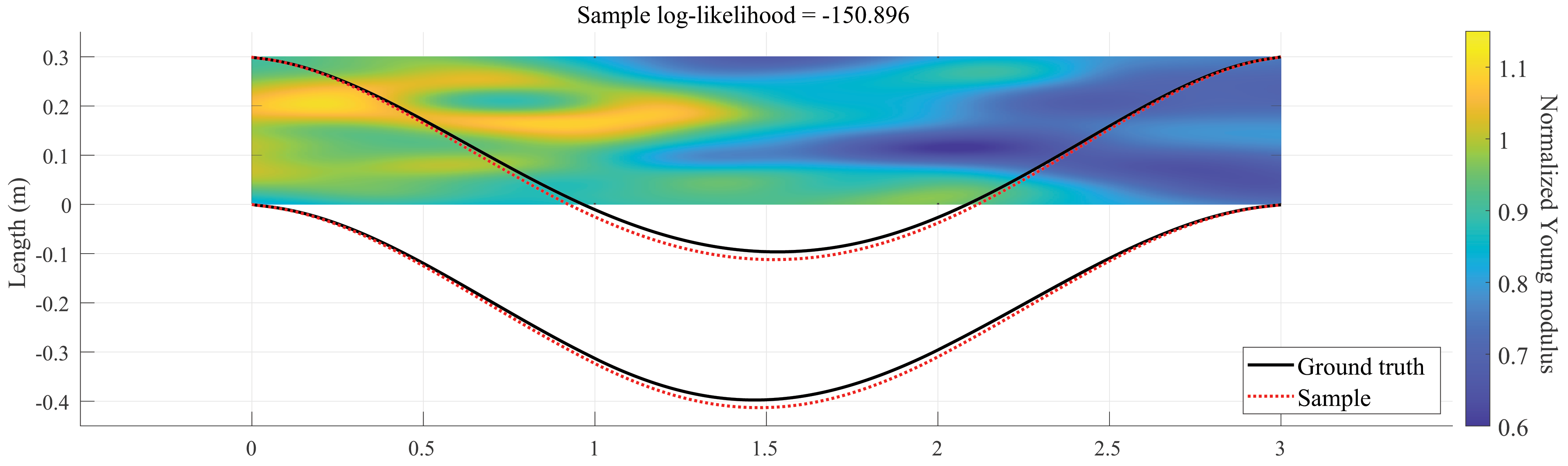}
\end{subfigure}
\begin{subfigure}[b]{0.9\textwidth}
\includegraphics[width=\textwidth]{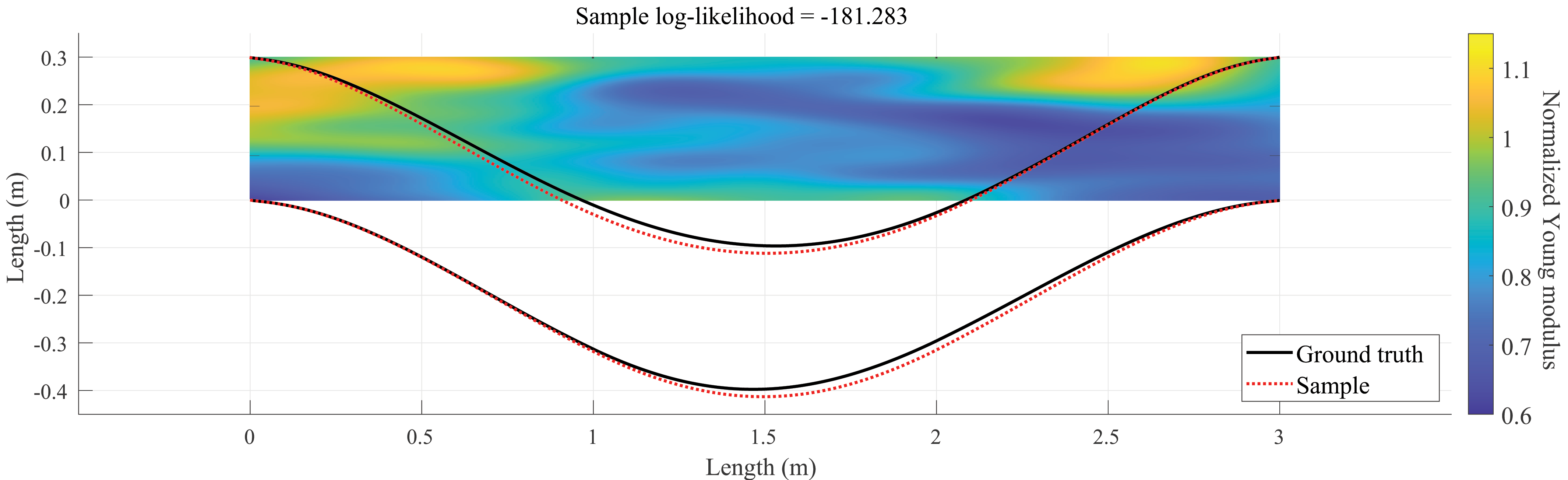}
\end{subfigure}
\caption{Normalized values of Young's modulus of the beam material for four different samples generated at the finest level of the multilevel MCMC algorithm, with corresponding log-likelihood scores and displacement observations for each of the samples (red) and for ground truth stiffness (black). Displacements are amplified by a factor $5\cdot 10^4$ for visibility.}
\label{fig:samples}
\end{figure}

\section{Conclusion}

We show how the existing multilevel MCMC methodology, which has previously been used effectively for high-resolution models with nevertheless low-resolution discrete observations, can be extended to fit in the case where both the model and the data are of very high resolution. We do this by extending the level-dependent treatment of parameters and discretisations to also include the observations. We add a weighting function to the observations which reduces their dimensionality such that these weighted observations can be used as proxies to the true observations at coarser levels in the multilevel algorithm.

We show how the level-dependent treatment of observations can be included in the multilevel algorithm such that, under mild assumptions, there are no losses in efficiency compared to the level-independent case. This is supported by a numerical experiment, which illustrates that for very high-resolution data some form of reduction on the data is necessary to continue exploiting the computational efficiency of the coarser levels in the multilevel hierarchy.

Our convergence analysis focuses on a broad class of probability models, including any mildly bounded Lipschitz continuous transformation of normal fields with Mat\'ern covariance kernels of sufficient smoothness ($\nu \geq 1$). Convergence rates for these types of models are obtained and illustrated through a different numerical experiment. Finally, we show how the method reconstructs posterior estimates for a simple stiffness reduction model case. We show that the MCMC method can recover accurate posterior estimates, and point to limitations in the amount of information present in the data, making this approach well-suited for uncertainty quantification.

\section*{Acknowledgements}

The authors would like to thank Christiaan Mommeyer for his assistance with implementing the linear elasticity forward models in the Stabil toolbox for Matlab, Arne Bouillon for his useful remarks in proofreading the convergence analysis of Section \ref{ch:conv_analysis}, and the anonymous referees for their helpful comments that improved the quality of this work.

\bibliography{mylibrary}

\end{document}

%% file: ex_shared.tex
\usepackage{lipsum}
\usepackage{amsfonts}
\usepackage{graphicx}
\usepackage{epstopdf}
\usepackage{algorithmic}
\usepackage{amssymb}
\usepackage{todonotes}
\usepackage{subcaption}
\usepackage{enumitem}
\usepackage{hyperref}
\usepackage{subcaption}
\ifpdf
  \DeclareGraphicsExtensions{.eps,.pdf,.png,.jpg}
\else
  \DeclareGraphicsExtensions{.eps}
\fi

\usepackage{enumitem}
\setlist[enumerate]{leftmargin=.5in}
\setlist[itemize]{leftmargin=.5in}

\newsiamremark{remark}{Remark}
\newsiamremark{hypothesis}{Hypothesis}
\crefname{hypothesis}{Hypothesis}{Hypotheses}
\newsiamthm{claim}{Claim}

\headers{MLMCMC for high-resolution observations}{P. Vanmechelen, G. Lombaert and G. Samaey}

\title{Multilevel Markov Chain Monte Carlo with likelihood scaling \\ for Bayesian inversion with high-resolution observations
\thanks{Submitted to the editors January 26th, 2024. Revised version submitted November 6th, 2024.
\funding{P. Vanmechelen is an SB PhD fellow of the Research Foundation Flanders (FWO), funded by Grant No. 1SD1823N.}}}

\author{P. Vanmechelen\thanks{KU Leuven, Department of Computer Science.}
\and G. Lombaert\thanks{KU Leuven, Department of Civil Engineering.}
\and G. Samaey\footnotemark[2]}

\usepackage{amsopn}
